\documentclass[12pt]{article}
\usepackage{amsmath,latexsym,amssymb}
\usepackage{enumerate}
\usepackage{amsthm}
\usepackage{epsfig,graphicx}
\usepackage{epic}

\newcommand\NN{\mathbb{N}}
\newcommand\RR{\mathbb{R}}

\newcommand\HH{\mathcal{H}}
\newcommand\II{\mathcal{I}}

\newcommand\KK{\mathcal{K}}

\newcommand\eps{{\varepsilon}}

\renewcommand{\mod}{\operatorname{mod}}

\newtheorem{theorem}{Theorem}[section]

\newtheorem{corollary}[theorem]{Corollary}
\newtheorem{example}[theorem]{Example}
\newtheorem{lemma}[theorem]{Lemma}
\newtheorem{remark}[theorem]{Remark}
\newtheorem{proposition}[theorem]{Proposition}

\newtheorem{innercustomthm}{Example}
\newenvironment{customex}[1]
  {\renewcommand\theinnercustomthm{#1}\innercustomthm}
  {\endinnercustomthm}

\numberwithin{equation}{section}

\newcommand{\address}{Address: Department of Mathematics, University of North Texas, 1155 Union Circle \#311430, Denton, TX 76203-5017, USA; E-mail: allaart@unt.edu}

\title{Differentiability and H\"older spectra of a class of self-affine functions}
\author{Pieter C. Allaart \footnote{\address}}

\begin{document}

\maketitle

\begin{abstract}
This paper studies a large class of continuous functions $f:[0,1]\to\RR^d$ whose range is the attractor of an iterated function system $\{S_1,\dots,S_{m}\}$ consisting of similitudes. This class includes such classical examples as P\'olya's space-filling curves, the Riesz-Nagy singular functions and Okamoto's functions. The differentiability of $f$ is completely classified in terms of the contraction ratios of the maps $S_1,\dots,S_{m}$. Generalizing results of Lax (1973) and Okamoto (2006), it is shown that either (i) $f$ is nowhere differentiable; (ii) $f$ is non-differentiable almost everywhere but with uncountably many exceptions; or (iii) $f$ is differentiable almost everywhere but with uncountably many exceptions. The Hausdorff dimension of the exceptional sets in cases (ii) and (iii) above is calculated, and more generally, the complete multifractal spectrum of $f$ is determined.


\end{abstract}

\section{Introduction}

In 1973, P. Lax \cite{Lax} proved a remarkable theorem about the differentiability of P\'olya's space-filling curve, which maps a closed interval continuously onto a solid right triangle. Unlike the space-filling curves of Peano and Hilbert, which had been known to be nowhere differentiable, Lax found that the differentiability of the P\'olya curve depends on the value of the smallest acute angle $\theta$ of the triangle. (Roughly speaking, the larger the angle, the less differentiable the function is; see Example \ref{ex:Polya-ctd} below.) 

More than 30 years later, H. Okamoto \cite{Okamoto} introduced a one-parameter family of self-affine functions that includes the Cantor function as well as functions previously studied by Perkins \cite{Perkins} and Katsuura \cite{Katsuura}. Okamoto showed that the differentiability of his functions depends on the parameter $a\in(0,1)$ in much the same way as the differentiability of the P\'olya curve depends on the angle $\theta$ (though it is not clear whether Okamoto was aware of Lax's result). See Example \ref{ex:Okamoto-ctd} below.

While Okamoto's function and the P\'olya curve are not directly related, both can be viewed as special cases of a large class of self-affine functions. The aim of this article is to study the differentiability of this class of functions, thereby generalizing the results of Lax and Okamoto, and to determine their finer local regularity behavior in the form of the pointwise H\"older spectrum.

Our class of functions is a subclass of that considered in \cite{BKK} and may be described as follows. Fix $d\in\NN$, an integer $m\geq 2$, and points $\mathbf{a},\mathbf{b}\in\RR^d$ with $|\mathbf{a}-\mathbf{b}|=1$. (Without loss of generality we take $\mathbf{a}=(0,0,\dots,0)$ and $\mathbf{b}=(1,0,\dots,0)$.) Fix a vector $\boldsymbol\eps=(\eps_1,\dots,\eps_m)\in\{0,1\}^m$. Let $S_1,\dots,S_{m}$ be contractive similitudes in $\RR^d$ satisfying the ``connectivity conditions"
\begin{gather}
S_1\big((1-\eps_1)\mathbf{a}+\eps_1\mathbf{b}\big)=\mathbf{a}, \label{eq:first-fixed-point} \\
S_m\big(\eps_m\mathbf{a}+(1-\eps_m)\mathbf{b}\big)=\mathbf{b}, \label{eq:last-fixed-point} \\
S_{i-1}\big(\eps_{i-1}\mathbf{a}+(1-\eps_{i-1})\mathbf{b}\big)=S_i((1-\eps_i)\mathbf{a}+\eps_i\mathbf{b}\big), \qquad i=2,\dots,m. \label{eq:connecting-pieces}
\end{gather}
Put $\lambda_i:=\mathrm{Lip}(S_i)$. If $m\geq 3$, we allow one or more of the $S_i$ to be constant, so $\lambda_i=0$.

Let $c_1,\dots,c_m$ be positive numbers with $\sum_{i=1}^m c_i=1$. Put $\sigma_i:=\sum_{j=1}^{i-1}c_j+\eps_i c_i$ for $i,\dots,m$, and define the maps
\begin{equation*}
\phi_i(t):=(-1)^{\eps_i}c_i t+\sigma_i, \qquad i=1,\dots,m,
\end{equation*}
so $\phi_i$ maps $[0,1]$ linearly onto a closed interval $I_i$ of length $c_i$, and the intervals $I_1,\dots,I_m$ are nonoverlapping with $\bigcup_{i=1}^m I_i=[0,1]$. By a theorem of de Rham \cite{deRham}, there exists a unique continuous function $f:[0,1]\to\RR^d$ satisfying the functional equation
\begin{equation}
f(t)=S_i\big(f(\phi_i^{-1}(t)\big), \qquad t\in I_i , \qquad i=1,\dots,m.
\label{eq:our-self-affine-functions-intro}
\end{equation}
Following \cite{BKK}, we shall call $\boldsymbol\eps$ the {\em signature} of $f$.
The image $\Gamma:=f([0,1])$ is a connected, self-similar curve in $\RR^d$ satisfying $\Gamma=\bigcup_{i=1}^m S_i(\Gamma)$. Note that \eqref{eq:first-fixed-point}-\eqref{eq:connecting-pieces} imply that $\sum_{i=1}^{m}\lambda_i\geq 1$. To avoid degenerate cases, we shall assume throughout that
\begin{equation}
(\lambda_1,\dots,\lambda_{m})\neq (c_1,\dots,c_m).
\label{eq:lambda-assumption}
\end{equation}

\begin{remark}
{\rm
Let $\KK_\phi$ be the unique nonempty compact subset of $[0,1]$ such that
\begin{equation}
\KK_\phi=\bigcup_{i:\lambda_i>0}\phi_i(\KK_\phi).
\label{eq:set-equation}
\end{equation}
Since $f$ is constant on each of the intervals making up the complement of $\KK_\phi$, we can think of $f$ alternatively as a continuous function from the self-similar set $\KK_\phi$ in $[0,1]$ onto the self-similar set $\Gamma$ in $\RR^d$.
}
\end{remark}

In each of the examples below, we take $(c_1,\dots,c_m)=(1/m,\dots,1/m)$ and $\boldsymbol\eps=(0,0,\dots,0)$, unless otherwise specified.

\begin{example}[The P\'olya curve] \label{ex:Polya}
{\rm
Take $d=m=2$. Let $\Delta$ be a right triangle positioned as in Figure \ref{fig:triangle}. We assume that $\theta$, the smaller of the two acute angles of $\Delta$, is the angle at $(1,0)$. The two subtriangles $\Delta_1$ and $\Delta_2$ in the figure are similar to $\Delta$; let $S_1$ and $S_2$ be the affine transformations which map $\Delta$ onto $\Delta_1$ and $\Delta_2$, respectively. The function $f$ determined by \eqref{eq:our-self-affine-functions-intro} in this case is P\'olya's space-filling curve \cite{Polya}, which maps the interval $[0,1]$ onto the triangle $\Delta$, and $(\lambda_1,\lambda_2)=(\sin\theta,\cos\theta)$.
}
\end{example}

\begin{figure}[h]
\begin{center}
\begin{picture}(100,65)(0,-15)
    \put(0,0){\line(2,3){37}}
    \put(37,56){\line(3,-2){85}}
    \put(37,56){\line(0,-1){56}}
    \put(0,0){\line(1,0){122}}
    \put(18,18){\makebox(0,0)[tl]{$\Delta_1$}}
    \put(60,18){\makebox(0,0)[tl]{$\Delta_2$}}
    \put(80,50){\makebox(0,0)[tl]{$\Delta$}}
		\qbezier(108,0)(108,4)(110,7)
		\put(95,11){\makebox(0,0)[tl]{$\theta$}}
		\put(-13,-5){\makebox(0,0)[tl]{$(0,0)$}}
		\put(110,-5){\makebox(0,0)[tl]{$(1,0)$}}
    \end{picture}
\caption{\footnotesize The right triangle $\Delta$, and its similar subtriangles $\Delta_1$ and $\Delta_2$.}
\label{fig:triangle}
\end{center}
\end{figure}
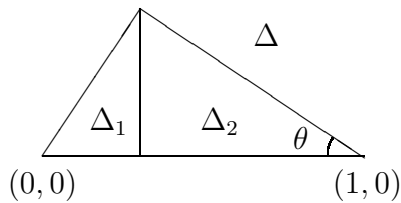

\begin{example}[Okamoto's functions] \label{ex:Okamoto}
{\rm
Take $d=1$ and $m=3$. Fix a parameter $a\in(0,1)$, and set $S_1(x)=ax$, $S_2(x)=a+(1-2a)x$ and $S_3(x)=ax+1-a$. Then $f$ determined by \eqref{eq:our-self-affine-functions-intro} is Okamoto's function \cite{Okamoto}, shown in Figure \ref{fig:construction}. Note that $a=1/2$ gives the Cantor function. The special cases $a=5/6$ and $a=2/3$ had been considered previously by Perkins \cite{Perkins} and Katsuura \cite{Katsuura}, respectively.
}
\end{example}

\begin{figure}[h]
\begin{center}
\begin{picture}(160,150)(0,5)
\put(30,20){\line(1,0){126}}
\put(30,20){\line(0,1){126}}
\put(30,146){\line(1,0){126}}
\put(156,20){\line(0,1){126}}
\put(27,15){\makebox(0,0)[tl]{$0$}}
\put(72,18){\line(0,1){4}}
\put(62,15){\makebox(0,0)[tl]{$1/3$}}
\put(114,18){\line(0,1){4}}
\put(104,15){\makebox(0,0)[tl]{$2/3$}}
\put(154,15){\makebox(0,0)[tl]{$1$}}
\put(17,25){\makebox(0,0)[tl]{$0$}}
\put(28,104){\line(1,0){4}}
\put(16,107){\makebox(0,0)[tl]{$a$}}
\put(28,62){\line(1,0){4}}
\put(-2,67){\makebox(0,0)[tl]{$1-a$}}
\put(19,150){\makebox(0,0)[tl]{$1$}}
\put(30,20){\line(1,2){42}}
\put(72,104){\line(1,-1){42}}
\put(114,62){\line(1,2){42}}
\thicklines
\dottedline{4}(30,20)(156,146)
\thinlines
\end{picture}
\qquad
\epsfig{file=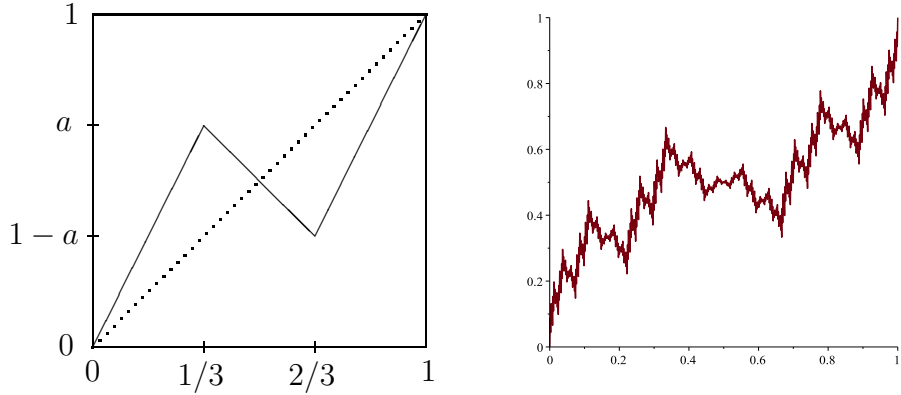, height=.25\textheight, width=.35\textwidth}
\end{center}
\caption{The generating pattern and graph of Okamoto's function, for $a=2/3$.}
\label{fig:construction}
\end{figure}


\begin{example}[The Riesz-Nagy function] \label{ex:Riesz-Nagy}
{\rm
Take $d=1$ and $m=2$, and fix a parameter $a\in(0,1)$, $a\neq 1/2$. Setting $S_1(x)=ax$ and $S_2(x)=a+(1-a)x$, we obtain the Riesz-Nagy function \cite{Riesz-Nagy,Salem}, one of the best known examples of a strictly increasing function whose derivative is almost everywhere zero; see Figure \ref{fig:Riesz-Nagy function}(a). 
In Section \ref{sec:CMT}, the Riesz-Nagy functions will serve as time subordinators for other functions of the form \eqref{eq:our-self-affine-functions-intro} with $m=2$.
}
\end{example}

\begin{figure}
\begin{center}
\vspace{-.2\textheight}
(a)\epsfig{bbllx=0,bblly=630,bburx=550,bbury=1230,file=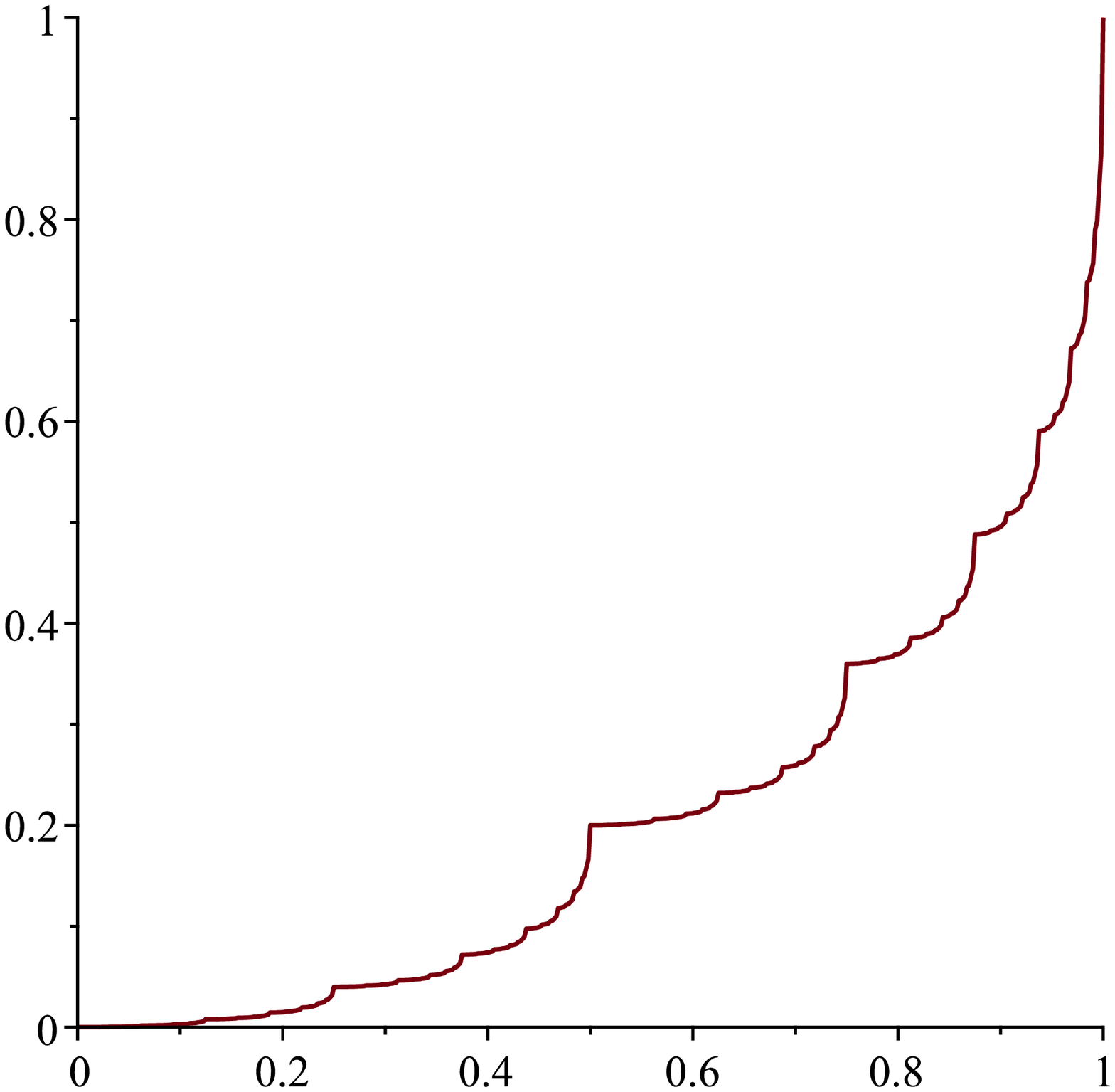, height=.25\textheight, width=.35\textwidth} \qquad\qquad
(b)\epsfig{bbllx=0,bblly=630,bburx=550,bbury=1230,file=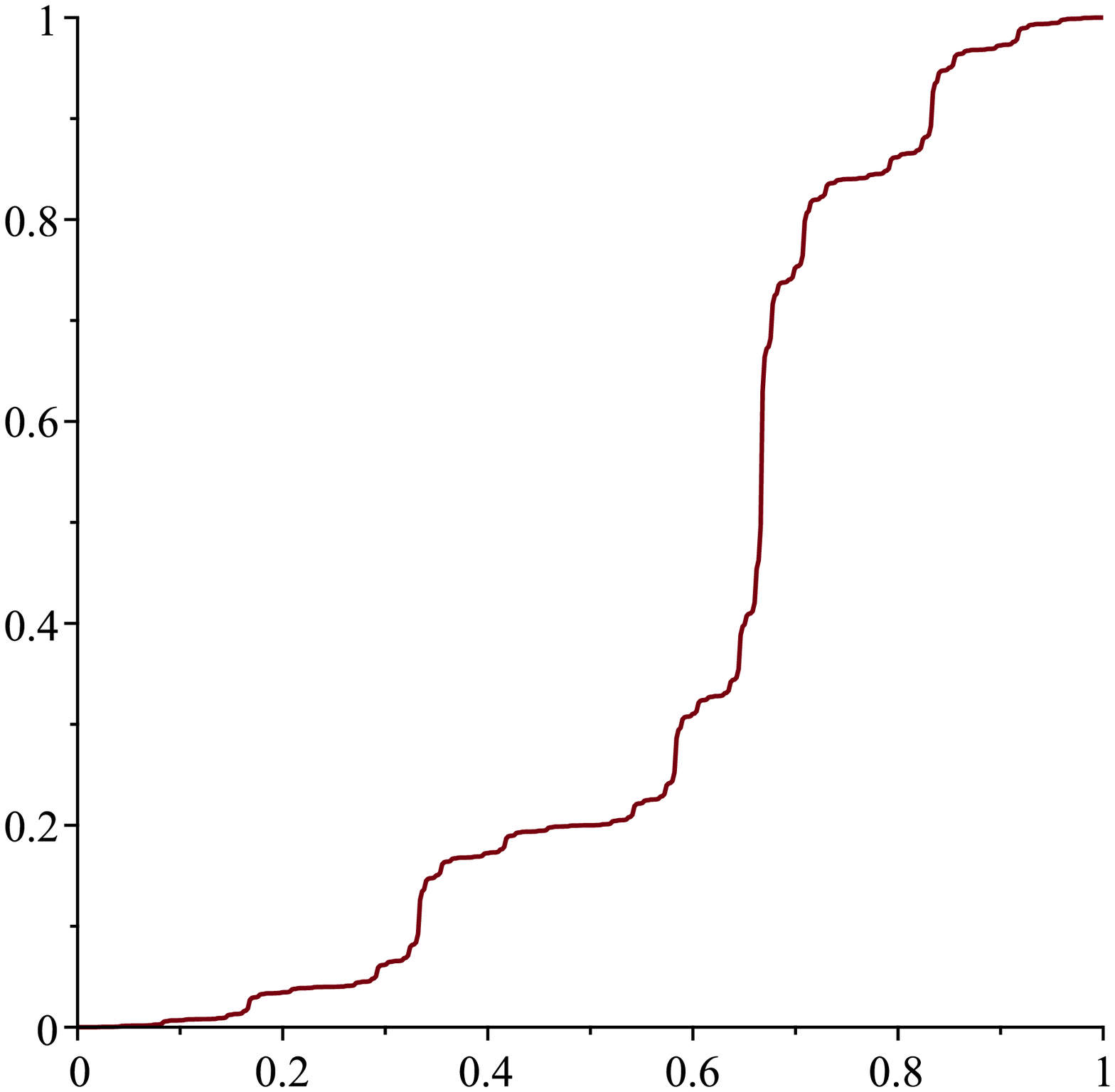, height=.25\textheight, width=.35\textwidth}
\end{center}
\vspace{.19\textheight}
\caption{The Riesz-Nagy function (left) and Gray code singular function (right).}
\label{fig:Riesz-Nagy function}
\end{figure}

\begin{example}[Gray code singular function] \label{ex:Gray}
{\rm
Take again $d=1$ and $m=2$. Let $c_1=c_2=1/2$, $\eps_1=0$ and $\eps_2=1$ (so $\phi_1(t)=t/2$ and $\phi_2(t)=1-t/2$). Fix $a\in(0,1)$, $a\neq 1/2$, and put $S_1(x)=ax$ and $S_2(x)=1-(1-a)x$. The function $f$ obtained this way (see Figure \ref{fig:Riesz-Nagy function}(b)) is the {\em Gray code singular function}, introduced in \cite{Kobayashi-Z} and so called because of its connection with the Gray code representation of real numbers.
}
\end{example}

\begin{example}[Distribution functions of self-similar measures] \label{ex:self-similar-measure}
{\rm
Generalizing the last two examples, let $J_1,\dots,J_k$ be nonoverlapping closed subintervals of $[0,1]$, ordered so that $J_i$ lies to the left of $J_j$ when $i<j$. For $1\leq j\leq k$, let $\psi_j$ be one of the two linear contractions which map $[0,1]$ onto $J_j$, and set $r_j:=\mathrm{Lip}(\psi_j)$. Let $(\pi_1,\dots,\pi_k)$ be a probability vector with $\pi_j>0$ for each $j$. There is then a unique nonempty compact set $F\subset[0,1]$ such that $F=\psi_1(F)\cup\dots\cup\psi_k(F)$, and there is a unique probability measure $\mu$ concentrated on $F$ such that
\begin{equation}
\mu=\sum_{j=1}^k \pi_j \mu\circ\psi_j^{-1}.
\end{equation}
Let $f(t):=\mu([0,t])$ for $t\in[0,1]$; then $f$ is of the form \eqref{eq:our-self-affine-functions-intro}. To determine the parameters, write $J_j=[s_j,t_j]$, $j=1,\dots,k$, and set $t_0:=0$ and $s_{k+1}:=1$. Let $k':=\#\{j\in\{1,\dots,k+1\}:t_{j-1}<s_j\}$. The points $s_1,t_1,\dots,s_k,t_k$ (not all necessarily distinct) divide $[0,1]$ into $m:=k+k'$ nonoverlapping closed subintervals; let us label them $I_1,\dots,I_m$ from left to right. For $i=1,\dots,m$, set $\lambda_i:=\pi_j$ if $I_i=J_j$ for some $j$, and else set $\lambda_i=0$. Set $c_i:=|I_i|$, so $c_i=r_j$ if $I_i=J_j$ for some $j$, and otherwise $c_i$ is the length of a ``gap" between two successive intervals $J_{j-1}$ and $J_j$. Note that this naturally yields examples of our set-up with some of the $\lambda_i$ equal to zero. In this case, $\sum_{i=1}^m\lambda_i=\sum_{j=1}^k\pi_j=1$, and $\KK_\phi=F$.
}
\end{example}

Other examples of functions satisfying \eqref{eq:our-self-affine-functions-intro} include the space-filling curves of Peano ($m=9$) and Hilbert ($m=4$), and classical fractals such as the Koch curve and the L\'evy curve \cite{Levy}, as well as asymmetric versions of these. However, as most of these functions are nowhere differentiable and monofractal, they are less interesting from the point of view of this article. For a comprehensive survey of space-filling curves, see \cite{Sagan}.

Section \ref{sec:results} outlines the main results of the paper, illustrating them for some of the above examples. Surprisingly, the differentiability of $f$ depends on the maps $S_1,\dots,S_{m}$ only through their contraction ratios $\lambda_1,\dots,\lambda_{m}$. This means that, especially when $d\geq 2$, there are many different functions in our class with the same differentiability structure, and even with the same pointwise H\"older spectrum.

We show first that the only possible finite derivative of a function $f$ of the form \eqref{eq:our-self-affine-functions-intro} under the assumption \eqref{eq:lambda-assumption} is zero. We then generalize Lax's and Okamoto's theorems by showing that, depending on the values of $\lambda_1,\dots,\lambda_{m}$ and $c_1,\dots,c_m$, $f$ is either (i) nowhere differentiable; (ii) differentiable almost nowhere, with uncountably many exceptions; or (iii) differentiable almost everywhere, with uncountably many exceptions. In cases (ii) and (iii), we compute the Hausdorff dimension of the exceptional sets. For example, in the case of P\'olya's space-filling curve we obtain that, if $15^\circ\leq\theta<30^\circ$, the set of points where $f$ is differentiable has Hausdorff dimension $-p\log_2 p-(1-p)\log_2(1-p)$, where
\begin{equation*}
p=\frac{\log(2\cos\theta)}{\log(\cot\theta)}.
\end{equation*}

A large part of the paper is devoted to the pointwise H\"older spectrum -- or multifractal spectrum -- of $f$. Globally, a function $f$ is said to be H\"older continuous with exponent $\alpha>0$ if there is a constant $C$ such that $|f(x)-f(y)|\leq C|x-y|^\alpha$ for all $x$ and $y$. 
However, this $\alpha$ represents the ``worst possible" behavior, and in general, a continuous function can at many points have substantially better regularity than the worst case. 

Consider first the case of a function $f:[0,1]\to\RR$. For $\alpha>0$ and $t_0\in(0,1)$, write $f\in C^\alpha(t_0)$ if there is a constant $C$ and a polynomial $P$ of degree less than $\alpha$ such that
\begin{equation}
|f(t)-P(t-t_0)|\leq C|t-t_0|^\alpha \qquad\mbox{for all $t\in(0,1)$}.
\label{eq:Holder-def}
\end{equation}
The {\em pointwise H\"older exponent} of $f$ at $t$ is the number
\begin{equation}
\alpha_f(t):=\sup\{\alpha>0: f\in C^\alpha(t)\}, \qquad t\in(0,1),
\label{eq:Holder-exponent-definition}
\end{equation}
and the {\em pointwise H\"older spectrum} of $f$ is the function $\alpha \mapsto \dim_H E_f(\alpha)$, where
\begin{equation*}
E_f(\alpha):=\{t\in(0,1):\alpha_f(t)=\alpha\}, \qquad \alpha>0,
\end{equation*}
and $\dim_H$ denotes Hausdorff dimension.
For a function $f:[0,1]\to\RR^d$, one replaces the polynomial $P$ in \eqref{eq:Holder-def} by a $d$-tuple $P=(P_1,\dots,P_d)$ of polynomials, each of degree less than $\alpha$.

Local H\"older exponents can be difficult to calculate, especially for $\alpha>2$, where, in order to show that $f\not\in C^\alpha(t_0)$, one must prove that no polynomial satisfying \eqref{eq:Holder-def} exists. For this reason perhaps, many authors (e.g. \cite{BKK,Bedford}) use the following, simpler definition of H\"older exponent. Write $f\in \tilde{C}^\alpha(t_0)$ if there is a constant $C$ such that
\begin{equation*}
|f(t)-f(t_0)|\leq C|t-t_0|^\alpha \qquad\mbox{for all $t\in(0,1)$},
\end{equation*}
that is, the polynomial $P$ in \eqref{eq:Holder-def} is constant with value $f(t_0)$. Define
\begin{equation}
\tilde{\alpha}_f(t):=\sup\{\alpha>0: f\in \tilde{C}^\alpha(t)\}, \qquad t\in(0,1),
\label{eq:nondirectional-Holder-exponent}
\end{equation}
and let
\begin{equation*}
\tilde{E}_f(\alpha):=\{t\in(0,1): \tilde{\alpha}_f(t)=\alpha\}, \qquad \alpha>0.
\end{equation*}
We shall call $\tilde{\alpha}_f(t)$ the {\em nondirectional H\"older exponent} of $f$ at $t$, and refer to the function $\alpha\mapsto \dim_H\tilde{E}_f(\alpha)$ as the {\em nondirectional H\"older spectrum} of $f$. Observe that $\alpha_f(t)\geq\tilde{\alpha}_f(t)$, but the reverse inequality may fail in general. For example, if $f(t)=t^2$, then $\alpha_f(0)=\infty$ (since one can take $P(t)=t^2$ in \eqref{eq:Holder-def} for all $\alpha>2$), but $\tilde{\alpha}_f(0)=2$. In addition, a desirable property of H\"older exponents is that they are left unchanged upon perturbation of $f$ by a smooth function $g$. Indeed, $\alpha_f(t)=\alpha_{f+g}(t)$, whereas $\tilde{\alpha}_f(t)\neq\tilde{\alpha}_{f+g}(t)$ in general.

H\"older spectra are an important analytical tool in the study of certain physical processes that exhibit a wide range of local regularity behavior, such as intermittent turbulence flows or intensity of seismic waves; see \cite{Frisch,Mandelbrot}. They were studied by Jaffard \cite{Jaffard1,Jaffard2} for a large class of self-similar functions using wavelet methods. Jaffard's work assumes a certain smoothness condition which our functions do not satisfy, but Jaffard and Mandelbrot \cite{JafMan} later modified the wavelet approach to compute the H\"older spectrum of the P\'olya curve. Unfortunately, their proof omits some critical details and the final expression is incorrect. Ben Slimane \cite{BenSlimane} evaluates the multifractal spectrum of a family of self-similar functions based on binary splitting of the unit interval, which includes the Riesz-Nagy function. Both \cite{JafMan} and \cite{BenSlimane} use the Schauder basis, which is ideally suited to the case $m=2$. But it is less clear how to identify a suitable wavelet basis for $m\geq 3$, and moreover, the theorems underlying the wavelet method are rather technical. By contrast, our approach here, while not without technicalities, is completely elementary.

Another relevant paper, by Seuret \cite{Seuret}, uses an associated multinomial measure to compute the pointwise H\"older spectrum of Okamoto's function. His final expression too is incorrect, due to some unfortunate transcription errors. More importantly, Seuret does not carefully address the subtlety of H\"older exponents greater than one, where the polynomial $P$ in \eqref{eq:Holder-def} might be of higher degree; that is, he seems to assume without proof that $\alpha_f(t)=\tilde{\alpha}_f(t)$. We will show that this is indeed the case for Okamoto's function, and more generally, for all $f$ of the form \eqref{eq:our-self-affine-functions-intro} provided that $c_1=\dots=c_m=1/m$ and $\boldsymbol\eps=(0,0,\dots,0)$. We conjecture that $\alpha_f(t)=\tilde{\alpha}_f(t)$ regardless of the values of the $c_i$ and $\eps_i$. 

Theorem \ref{thm:Holder-spectrum} gives the nondirectonal H\"older spectrum of $f$, which is shown to satisfy the classical multifractal formalism. This is established by first obtaining an expression for $\tilde{\alpha}_f(t)$ at any point $t$, which seems interesting in its own right. A crucial tool in the proof of Theorem \ref{thm:Holder-spectrum} is the duality principle formulated in Proposition \ref{prop:duality}.

In Section \ref{sec:multifractal} we apply our result to the multifractal spectrum of the self-similar measures $\mu$ from Example \ref{ex:self-similar-measure}. We refine the classical multifractal formalism by showing that it holds also for the lower density of $\mu$. 

The final section of the paper connects our work to that of Seuret \cite{Seuret}, by showing that all functions of the form \eqref{eq:our-self-affine-functions-intro} can be written as a composition of a monofractal function and an increasing function, or time subordinator.

While the work for this paper was undertaken, a closely related article by B\'ar\'any et al.~\cite{BKK} appeared on the {\em arXiv}. That paper considers a more general setup, in which the maps $S_1,\dots,S_{m}$ are arbitrary affine contractions on $\RR^d$. While \cite{BKK} is quite general and technically sophisticated, our restriction here to similitudes offers several advantages: (i) B\'ar\'any et al.~define the pointwise H\"older exponent to be $\tilde{\alpha}_f(t)$, rather than $\alpha_f(t)$. While we are able to show that both definitions are equivalent for some functions of the form \eqref{eq:our-self-affine-functions-intro}, this is far from clear for the larger class of functions in \cite{BKK}; (ii) The main results of \cite{BKK} require that the maps $S_1,\dots,S_{m}$ satisfy a certain positivity condition which rules out many interesting examples, including the P\'olya curve; (iii) The authors of \cite{BKK} succeed only in determining the ``upper half" of the nondirectional H\"older spectrum. In order to obtain the full spectrum they need an additional and rather restrictive quasi-symmetry condition, which in our setting reduces to $\log\lambda_1/\log c_1=\log\lambda_m/\log c_m$. By focusing exclusively on similitudes, we obtain the full multifractal spectrum without having to make such a symmetry assumption; (iv) Our results are more explicit, and are obtained using elementary methods. The price to pay is, of course, that our results do not cover functions such as the main example in \cite{BKK}, a curve introduced by de Rham. Thus, it seems that the present article and \cite{BKK} complement each other quite well.

\section{Main results} \label{sec:results}

In what follows, we shall consider $f:[0,1]\to\RR^d$ to be differentiable at $t\in(0,1)$ if it has a well-defined {\em finite} derivative at $t$. From now on it will be assumed without further mention that $f$ is defined by \eqref{eq:our-self-affine-functions-intro} and that \eqref{eq:lambda-assumption} holds.

\begin{proposition} \label{prop:zero-or-none}
If $f$ is differentiable at a point $t$, then $f'(t)=0$.
\end{proposition}

Our first main result shows that the differentiability of $f$ is completely determined by the contraction ratios $\lambda_1,\dots,\lambda_{m}$ and $c_1,\dots,c_m$.

\begin{theorem} \label{thm:differentiability}
\begin{enumerate}[(i)]
\item If $\lambda_i\geq c_i$ for each $i$, then $f$ is nowhere differentiable;
\item If $\lambda_i<c_i$ for at least one $i$ but $\sum_{i=1}^{m}c_i\log(\lambda_i/c_i)\geq 0$, then $f$ is nondifferentiable almost everywhere but $f'(t)=0$ at uncountably many points;
\item If $\sum_{i=1}^{m}c_i\log(\lambda_i/c_i)<0$, then $f'(t)=0$ almost everywhere but $f$ is nondifferentiable at uncountably many points.
\end{enumerate}
\end{theorem}

\begin{example} \label{ex:Polya-ctd}
{\rm
Applying Theorem \ref{thm:differentiability} to the P\'olya curve from Example \ref{ex:Polya} and using the identity $2\sin\theta\cos\theta=\sin 2\theta$, we recover Lax's theorem, namely: (i) $f$ is nowhere differentiable when $\theta\geq 30^\circ$; (ii) $f$ is nondifferentiable almost everywhere but differentiable at uncountably many points when $15^\circ\leq\theta<30^\circ$; and (iii) $f$ is differentiable almost everywhere but nondifferentiable at uncountably many points when $\theta<15^\circ$. (We remark that Lax excluded the boundary cases $\theta=30^\circ$ and $\theta=15^\circ$ from his analysis; these were later dealt with by Bumby \cite{Bumby}.)
}
\end{example}

\begin{example} \label{ex:Okamoto-ctd}
{\rm
Let $f$ be Okamoto's function (Example \ref{ex:Okamoto}). Observe that $f$ is strictly increasing, and hence differentiable almost everywhere, when $a<1/2$. When $a>1/2$, we have $(\lambda_1,\lambda_2,\lambda_3)=(a,2a-1,a)$. Solving $\prod_{i=1}^3(3\lambda_i)=27a^2(2a-1)=1$ gives $a=a_0\approx .5592$. We now obtain Okamoto's result \cite{Okamoto}: (i) $f$ is nowhere differentiable when $a\geq 2/3$; (ii) $f$ is nondifferentiable almost everywhere but differentiable at uncountably many points when $a_0\leq a<2/3$; and (iii) $f$ is differentiable almost everywhere but nondifferentiable at uncountably many points when $a<a_0$. (We remark that Okamoto did not address the boundary case $a=a_0$, which was later settled by Kobayashi \cite{Kobayashi}.)
}
\end{example}

A natural next question is: What is the Hausdorff dimension of the exceptional sets in cases (ii) and (iii) of Theorem \ref{thm:differentiability}? Let
\begin{equation*}
\mathcal{D}(f):=\{t\in(0,1):f'(t)=0\}, \qquad \mathcal{D}_\sim(f):=[0,1]\backslash \mathcal{D}(f).
\end{equation*}
We will consider the dimensions of $\mathcal{D}(f)$ and $\mathcal{D}_\sim(f)$ in the context of the pointwise H\"older spectrum of $f$. Recall the definitions of $\alpha_f(t)$ and $\tilde{\alpha}_f(t)$ from \eqref{eq:Holder-exponent-definition} and \eqref{eq:nondirectional-Holder-exponent}. We first show that at least in the simplest cases, these two H\"older exponents are the same:

\begin{theorem} \label{thm:simple-Holder-spectrum}
Assume that $c_i=1/m$ for $i=1,\dots,m$, and that $\boldsymbol{\eps}=(0,0,\dots,0)$. Then $\alpha_f(t)=\tilde{\alpha}_f(t)$ for every $t\in(0,1)$.
\end{theorem}

Unfortunately, the author has been unable to extend this theorem to all functions of the form \eqref{eq:our-self-affine-functions-intro}. For the remainder of this section, we will therefore focus on the (easier to analyze) nondirectional H\"older spectrum of $f$; that is, the function $\alpha\mapsto \dim_H\tilde{E}_f(\alpha)$. 

First, we need some additional notation. Let 
\begin{equation*}
\mathcal{I}:=\{1,2,\dots,m\}, \qquad \mathcal{I}_0:=\{i\in\mathcal{I}:\lambda_i=0\}, \qquad \mathcal{I}_+:=\{i\in\mathcal{I}:\lambda_i>0\}.
\end{equation*} 
Define
\begin{equation*}
\rho_i:=\frac{\log\lambda_i}{\log c_i}, \quad i\in\II_+,
\end{equation*}
and let
\begin{equation*}
\alpha_{\min}:=\min_{i\in\II_+}\rho_i, \qquad \alpha_{\max}:=\max_{i\in\II_+}\rho_i.
\end{equation*}
Furthermore, let $s_{\min}$, $s_{\max}$ and $\hat{s}$ be the nonnegative numbers satisfying
\begin{equation*}
\sum_{i: \rho_i=\alpha_{\min}}c_i^{s_{\min}}=1, \qquad \sum_{i: \rho_i=\alpha_{\max}}c_i^{s_{\max}}=1, \qquad \sum_{i\in\II_+}c_i^{\hat{s}}=1.
\end{equation*}
Since $\#\II_+\geq 2$, $\hat{s}>0$. Put
\begin{equation*}
\hat{\alpha}:=\frac{\sum_{i\in\II_+}c_i^{\hat{s}}\log\lambda_i}{\sum_{i\in\II_+}c_i^{\hat{s}}\log c_i}.
\end{equation*}
Note that $\alpha_{\min}\leq \hat{\alpha}\leq \alpha_{\max}$.

For each $q\in\RR$, let $\beta(q)$ be the unique real number such that
\begin{equation}
\sum_{i\in\II_+}\lambda_i^q c_i^{\beta(q)}=1.
\label{eq:scaling-equation}
\end{equation}
It is well known from multifractal theory (e.g. \cite[Chapter 17]{Falconer}) that the function $\beta(q)$ is strictly decreasing and convex, and its Legendre transform
\begin{equation*}
\beta^*(\alpha):=\inf_{q\in\RR}\{\alpha q+\beta(q)\}, \qquad \alpha>0
\end{equation*}
is strictly concave on the interval $[\alpha_{\min},\alpha_{\max}]$, and takes the value $-\infty$ outside this interval.

\begin{theorem} \label{thm:Hausdorff-dimension}
Assume $\lambda_i<c_i$ for at least one $i\in\II$.
\begin{enumerate}[(i)]
\item If $\sum_{i=1}^{m}c_i\log(\lambda_i/c_i)\geq 0$, then $\dim_H\mathcal{D}(f)=\beta^*(1)>0$;
\item If $\sum_{i\in\II_+}c_i^{\hat{s}}\log(\lambda_i/c_i)\geq 0$, then $\dim_H\mathcal{D}_\sim(f)=\hat{s}>0$;
\item If $\sum_{i\in\II_+}c_i^{\hat{s}}\log(\lambda_i/c_i)<0$, then $\dim_H\mathcal{D}_\sim(f)=\beta^*(1)>0$.
\end{enumerate}
\end{theorem}

\begin{theorem} \label{thm:Holder-spectrum}
\begin{enumerate}[(i)]
\item $\tilde{E}_f(\alpha)=\emptyset$ when $\alpha\not\in[\alpha_{\min},\alpha_{\max}]\cup\{\infty\}$;
\item $\tilde{E}_f(\infty)$ is empty if $\II_0=\emptyset$, and has Lebesgue measure one otherwise;
\item $\dim_H \tilde{E}_f(\alpha)=\beta^*(\alpha)$ for all $\alpha\in(\alpha_{\min},\alpha_{\max})$;
\item $\dim_H \tilde{E}_f(\alpha_{\min})=s_{\min}$, and $\dim_H \tilde{E}_f(\alpha_{\max})=s_{\max}$;
\item The maximum value of $\dim_H \tilde{E}_f(\alpha)$ over $[\alpha_{\min},\alpha_{\max}]$ is attained at $\hat{\alpha}$, and
$\dim_H \tilde{E}_f(\hat{\alpha})=\hat{s}$.
Moreover, if $\II_0=\emptyset$, then $\tilde{E}_f(\hat{\alpha})$ has Lebesgue measure one.
\end{enumerate}
\end{theorem}

\begin{remark}
{\rm
Let $\KK_\phi$ be the self-similar set defined by \eqref{eq:set-equation}. Then $\dim_H \KK_\phi=\hat{s}$ and $0<\HH^{\hat{s}}(\KK_\phi)<\infty$, where $\HH^s$ denotes $s$-dimensional Hausdorff measure. Suppose $\sum_{i\in\II_+}c_i^{\hat{s}}\log(\lambda_i/c_i)<0$. Then $\hat{\alpha}>1$, so Theorem \ref{thm:Hausdorff-dimension}(iii) implies $\dim_H\mathcal{D}_\sim(f)=\beta^*(1)<\beta^*(\hat{\alpha})=\hat{s}$. We conclude that $f$ is differentiable $\HH^{\hat{s}}$-almost everywhere on $\KK_\phi$ (and of course, $f$ is differentiable everywhere outside $\KK_\phi$ as well).
}
\end{remark}

\begin{remark}
{\rm
When $c_i=1/m$ for each $i$ and $\boldsymbol{\eps}=(0,0,\dots,0)$, we may replace $\tilde{E}_f$ with $E_f$ in Theorem \ref{thm:Holder-spectrum}, in view of Theorem \ref{thm:simple-Holder-spectrum}.
}
\end{remark}

Before illustrating the last two theorems, we present an alternative view of $\dim_H \tilde{E}_f(\alpha)$ that will be important for the proofs later, and that is sometimes more convenient for concrete computations. Define the function
\begin{equation}
H(p_1,\dots,p_{m}):=\frac{\sum_{i=1}^{m}p_i\log p_i}{\sum_{i=1}^{m}p_i\log c_i},
\label{eq:h-definition}
\end{equation}
where as usual, we set $0\log 0\equiv 0$. We denote by $\Delta_m$ the standard simplex in $\RR^m$:
\begin{equation*}
\Delta_m:=\left\{\mathbf{p}=(p_1,\dots,p_{m})\in \RR^m: p_i\geq 0\ \mbox{for each $i$ and}\ \sum_{i=1}^{m}p_i=1\right\}.
\end{equation*}
Let
\begin{equation*}
\Delta_m^0:=\{\mathbf{p}=(p_1,\dots,p_{m})\in\Delta_m: p_i=0\ \mbox{for $i\in\mathcal{I}_0$}\}.
\end{equation*}
The following equality generalizes the ``maximum entropy/minimum pressure" duality observed in \cite[Theorem 11]{BSS}. 

\begin{proposition} \label{prop:duality}
For each $\alpha\in[\alpha_{\min},\alpha_{\max}]$, we have
\begin{equation}
\beta^*(\alpha)
=\max\bigg\{H(\mathbf{p}): \mathbf{p}=(p_1,\dots,p_{m})\in\Delta_m^0,\ \sum_{i\in\mathcal{I_+}}p_i(\log\lambda_i-\alpha\log c_i)=0\bigg\}.
\label{eq:duality}
\end{equation}
\end{proposition}

Proposition \ref{prop:duality} is geometrically pleasing: it represents $\beta^*(\alpha)$ as the maximum value of $H$ over the intersection of a simplex with a hyperplane. This intersection is nonempty for $\alpha\in[\alpha_{\min},\alpha_{\max}]$, as is easy to see. The characterization is especially useful when $m=2$, in which case the intersection consists of a single point, and no maximization or minimization is necessary. In this case, solving the equations $p_1+p_2=1$ and $p_1(\log\lambda_1-\alpha\log c_1)+p_2(\log\lambda_2-\alpha\log c_2)=0$ gives
\begin{equation}
p_1=\frac{\alpha\log c_2-\log\lambda_2}{\log\lambda_1-\log\lambda_2-\alpha(\log c_1-\log c_2)}, \qquad p_2=1-p_1.
\label{eq:binary-case-proportions}
\end{equation}

Observe that when $c_1=c_2=1/2$, $p_1$ varies linearly as a function of $\alpha$, and takes the values $0$ and $1$ at the endpoints of $[\alpha_{\min},\alpha_{\max}]$. Let $H(u):=H(u,1-u)=-u\log_2 u-(1-u)\log_2(1-u)$. Since $H$ is symmetric, we see from Theorem \ref{thm:Holder-spectrum} and Proposition \ref{prop:duality} that
\begin{equation}
\dim_H E_f(\alpha)=H\left(\frac{\alpha-\alpha_{\min}}{\alpha_{\max}-\alpha_{\min}}\right), \qquad \alpha\in [\alpha_{\min},\alpha_{\max}].
\label{eq:scaling-Holder-spectrum}
\end{equation}
In other words, a change in the values of $\lambda_1$ and $\lambda_2$ results only in a horizontal scaling and translation of the H\"older spectrum, but does not affect its general shape. 

When $m\geq 3$, one can either compute $\beta^*(\alpha)$ by minimizing $\alpha q+\beta(q)$ over $q$, or one can apply the method of Lagrange multipliers to the constrained optimization problem in \eqref{eq:duality}. Both approaches have their challenges in practice: the former requires one to estimate $\beta(q)$ numerically first for every real $q$; and the latter entails solving a system of nonlinear equations in $p_1,\dots,p_m$.
 In the special case when $c_i=1/m$ for $i=1,\dots,m$, however, both methods quickly yield a fairly explicit answer. We then have simply
\begin{equation*}
\beta(q)=\log_m\sum_{i\in\mathcal{I}_+}\lambda_i^q,
\end{equation*}
and setting $\beta'(q)=-\alpha$ gives that $\alpha q+\beta(q)$ is minimized at the value of $q$ for which
\begin{equation}
\sum_{i\in\mathcal{I}_+}\lambda_i^q \log(m^\alpha\lambda_i)=0.
\label{eq:Holder-s-equation}
\end{equation}
(This $q$ exists and is unique, since the function $q\mapsto \sum_{i\in\mathcal{I}_+}(m^\alpha\lambda_i)^q \log(m^\alpha\lambda_i)$ is strictly increasing and tends to $-\infty$ as $q\to-\infty$, and to $+\infty$ as $q\to+\infty$, provided $\alpha\in(\alpha_{\min},\alpha_{\max})$.) Alternatively, the method of Lagrange multipliers yields that the constrained maximum in \eqref{eq:duality} is attained at the point $\mathbf{p}^*=(p_1^*,\dots,p_{m}^*)$ given by
\begin{equation*}
p_i^*=\frac{\lambda_i^q}{\sum_{j\in\mathcal{I}_+}\lambda_j^q}, \qquad i\in\mathcal{I}_+
\end{equation*}
and $p_i^*=0$ for $i\in\mathcal{I}_0$, with $q$ as in \eqref{eq:Holder-s-equation}, after which some further algebra gives
\begin{equation*}
\beta^*(\alpha)=H(p_1^*,\dots,p_{m}^*)=\alpha q+\log_m\sum_{i\in\mathcal{I}_+}\lambda_i^q.
\end{equation*}

\begin{customex}{\ref{ex:Polya-ctd} (continued)}
{\rm
For the P\'olya curve, $\alpha_{\min}=-\log_2\cos\theta$, $\alpha_{\max}=-\log_2\sin\theta$, and $\hat{\alpha}=-\frac12\log_2(\sin\theta\cos\theta)$. A computation based on \eqref{eq:scaling-Holder-spectrum} yields
\begin{align*}
\dim_H E_f(\alpha)&=\log_2(-\log\tan\theta)\\
&\qquad+\frac{\log(2^\alpha\cos\theta)\log_2\log(2^\alpha\cos\theta)-\log(2^\alpha\sin\theta)\log_2(-\log(2^\alpha\sin\theta))}{\log\tan\theta}
\end{align*}
for $\alpha\in [\alpha_{\min},\alpha_{\max}]$. This expression, graphed in Figure \ref{fig:Holder-spectra-examples}(a)  for $\theta=25^\circ$, can be taken to correct the one given in \cite{JafMan}.

Moreover, setting $\alpha=1$ in \eqref{eq:binary-case-proportions} gives
\begin{equation*}
p_1=\frac{\log(2\cos\theta)}{\log(\cot\theta)}.
\end{equation*}
This is increasing in $\theta$ on $0\leq\theta\leq 30^\circ$. Writing $H(p):=H(p,1-p)$, we find that for $15^\circ\leq \theta<30^\circ$, $\dim_H \mathcal{D}(f)=H(p_1)$; and for $0<\theta<15^\circ$, $\dim_H\mathcal{D}_\sim(f)=H(p_1)$.
}
\end{customex}

\begin{figure}
\begin{center}
\vspace{-.2\textheight}
(a)\epsfig{bbllx=0,bblly=630,bburx=550,bbury=1230,file=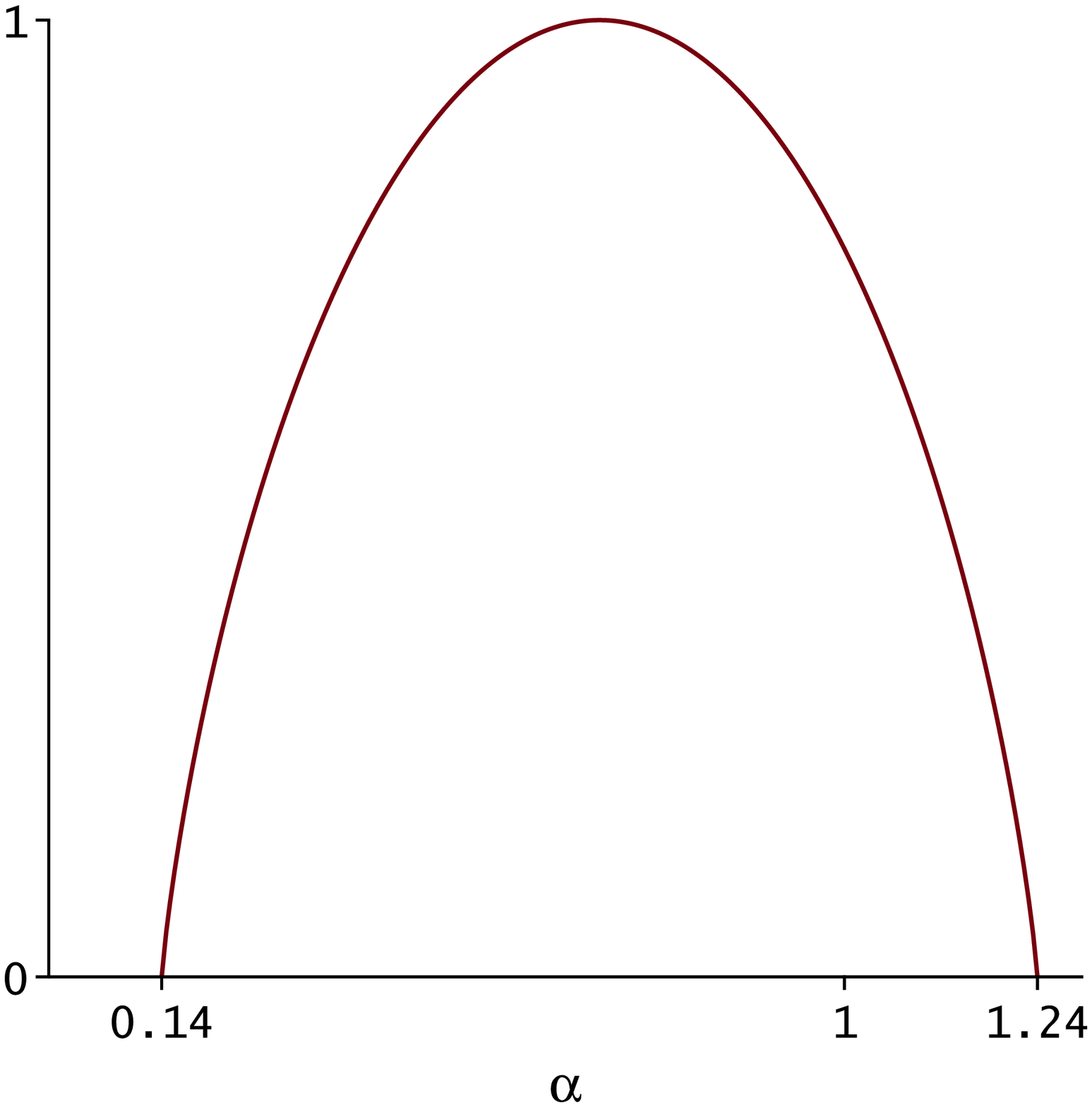, height=.25\textheight, width=.35\textwidth} \qquad\quad
(b)\hspace{-0.1in}\epsfig{bbllx=0,bblly=630,bburx=550,bbury=1230,file=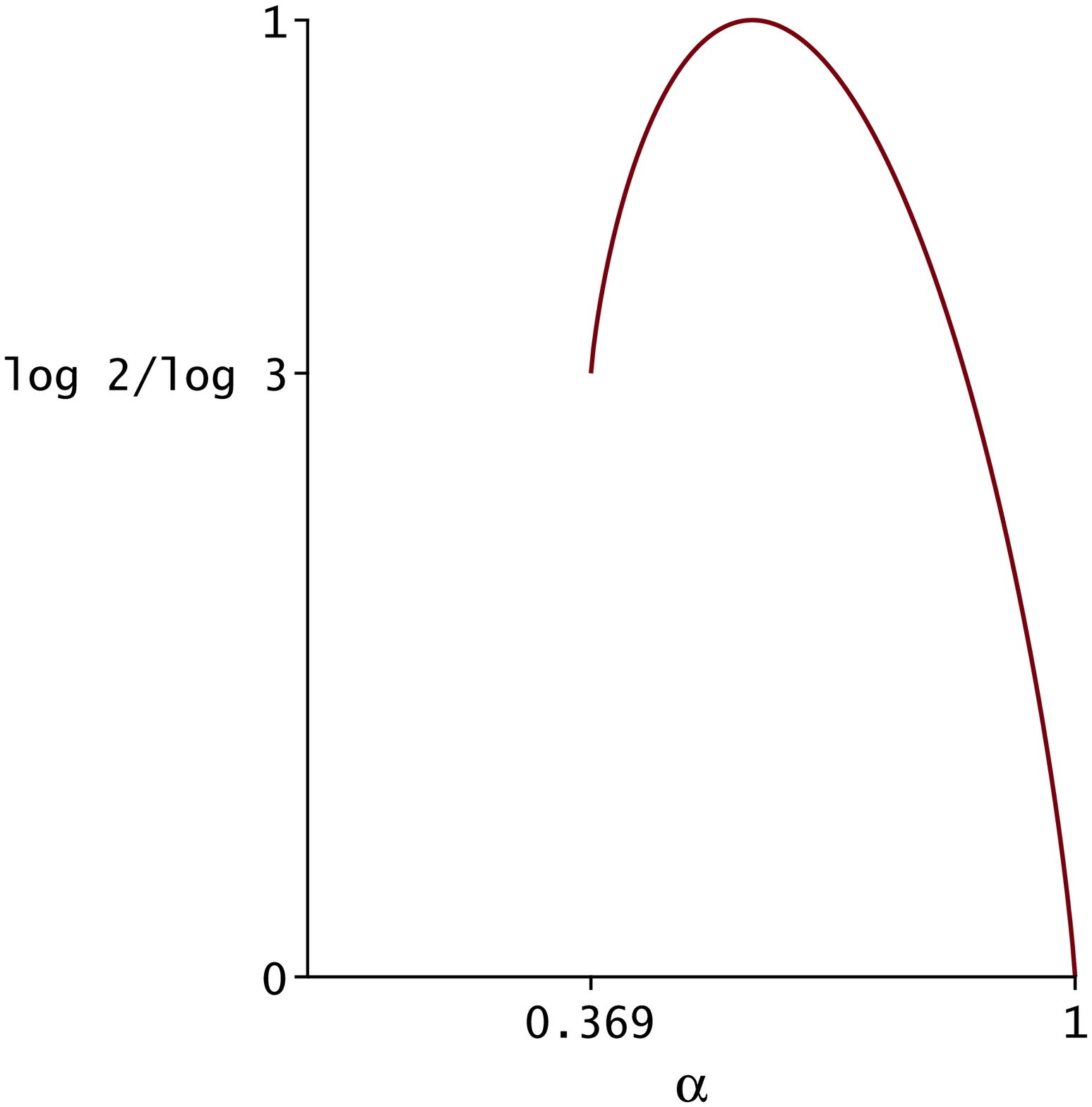, height=.25\textheight, width=.4\textwidth}
\end{center}
\vspace{.19\textheight}
\caption{Pointwise H\"older spectrum of P\'olya's space-filling curve (left, for $\theta=25^\circ$) and Okamoto's function (right, for $a=2/3$).}
\label{fig:Holder-spectra-examples}
\end{figure}

\begin{customex}{\ref{ex:Okamoto-ctd} (continued)}
{\rm
For the H\"older spectrum of Okamoto's function we consider three cases. First, if $a=1/2$, we have $\mathcal{I}_+=\{1,3\}$ and $\alpha_f(t)=\infty$ for every $t$ outside the ternary Cantor set $\mathcal{C}$ (and hence almost everywhere); while $\alpha_f(t)=\alpha_{\min}=\alpha_{\max}=\hat{\alpha}=\log_3 2$ for every $t\in\mathcal{C}$. Thus, $\dim_H E_f(\log_3 2)=\log_3 2$.

If $a\neq 1/2$ and $a>1/3$, then $\alpha_{\min}=-\log_3 a$ and $\alpha_{\max}=-\log_3|2a-1|$.
It is intuitively clear (and easy to check) that the constrained maximum in \eqref{eq:duality} must be obtained when $p_1=p_3$, since $\lambda_1=\lambda_3$. A straightforward calculation shows that for $\alpha_{\min}\leq\alpha\leq\alpha_{\max}$, 
\begin{equation}
\dim_H E_f(\alpha)=H(p_1^*,p_2^*,p_3^*)=-p^*\log_3 p^*-(1-p^*)\log_3\left(\frac{1-p^*}{2}\right),
\label{eq:Okamoto-Holder}
\end{equation}
where 
\begin{equation}
p^*:=p_2^*=\frac{\log a+\alpha\log 3}{\log a-\log|2a-1|},
\label{eq:proportion-of-ones}
\end{equation}
and $p_1^*=p_3^*=(1-p_2^*)/2$. See Figure \ref{fig:Holder-spectra-examples}(b). At the endpoints of the multifractal spectrum, we have $\dim_H E_f(\alpha_{\min})=\log_3 2$ and $\dim_H E_f(\alpha_{\max})=0$. However, $E_f(\alpha_{\max})$ is uncountably large: a closer inspection reveals that it consists of those points $t$ in whose ternary expansion the digit 1 has density 1, although $f\not\in C^{\alpha_{\max}}(t)$ for each such $t$.

Note that $p^*$ increases linearly from $0$ at $\alpha_{\min}$ to $1$ at $\alpha_{\max}$. Thus, the graph of $\dim_H E_f(\alpha)$ is the same for each $a>1/3, a\neq 1/2$, up to a horizontal scaling and translation. 

Finally, when $a<1/3$ the calculation is the same as in the second case above, but the endpoints are reversed: $\alpha_{\min}=-\log_3(1-2a)$, and $\alpha_{\max}=-\log_3 a$. Here the worst regularity ($\alpha_f(t)=\alpha_{\min}$) is achieved when the {\em upper} density of the digit 1 in the ternary expansion of $t$ is 1, so $E_f(\alpha_{\min})$ is uncountable but of Hausdorff dimension zero. The graph of $\dim_H E_f(\alpha)$ is the reverse of that in the case $a>1/3$.

Setting $\alpha=1$ in \eqref{eq:proportion-of-ones}, the right hand side of \eqref{eq:Okamoto-Holder} gives the Hausdorff dimension of $\mathcal{D}(f)$ when $a\geq a_0\approx .5592$; and of $\mathcal{D}_\sim(f)$ when $a<a_0$. This result was first reported in \cite{Allaart}, with a subsequent generalization in \cite{Allaart2}.
}
\end{customex}

\begin{example} \label{ex:Riesz-Nagy-ctd}
{\rm
Let $f$ be the Riesz-Nagy function with parameter $a$ (Example \ref{ex:Riesz-Nagy}) or the Gray code singular function with parameter $a$ (Example \ref{ex:Gray}). From \eqref{eq:scaling-Holder-spectrum} it follows that $\dim_H D_\sim(f)=H(p,1-p)$, where 
\begin{equation*}
p=\frac{\log 2a}{\log a-\log(1-a)}.
\end{equation*}
Likewise, the pointwise H\"older spectrum of $f$ may be obtained from \eqref{eq:scaling-Holder-spectrum}.
}
\end{example}

The remainder of this article is organized as follows. Section \ref{sec:notation} introduces notation and preliminary results, including a proof of Proposition \ref{prop:duality}. Proposition \ref{prop:zero-or-none} and Theorem \ref{thm:differentiability} are proved in Section \ref{sec:differentiability-proofs}, and Theorem \ref{thm:simple-Holder-spectrum} is proved in Section \ref{sec:simple-Holder}. Section \ref{sec:Holder-exponent-calculation} gives the computation of $\tilde{\alpha}_f(t)$ for every $t$, and Section \ref{sec:Holder-proofs} contains proofs of Theorems \ref{thm:Hausdorff-dimension} and \ref{thm:Holder-spectrum}. Section \ref{sec:multifractal} applies the main results to the multifractal spectrum of self-similar measures on $\RR$, and Section \ref{sec:CMT} shows how functions of the form \eqref{eq:our-self-affine-functions-intro} can be written as the composition of a monofractal function and an increasing function.

\section{Preliminaries} \label{sec:notation}

Let $\Omega:=\II^\NN$. For $\mathbf{i}=(i_1,i_2,\dots)\in\Omega$, the intersection
\begin{equation*}
\pi(\mathbf{i}):=\bigcap_{n=1}^\infty \phi_{i_1}\circ\dots\circ\phi_{i_n}([0,1])
\end{equation*}
consists of a single point. Call $\mathbf{i}$ a {\em coding} of $t\in[0,1]$ if $\pi(\mathbf{i})=\{t\}$. Each point $t\in[0,1]$ has at most two distinct codings; we shall call the lexicographically largest one the {\em standard coding} of $t$. We write $t\sim (i_1,i_2,\dots)$ to indicate that $(i_1,i_2,\dots)$ is the standard coding of $t$. In the special case when $c_i=1/m$ for each $i$ and $\boldsymbol{\eps}=(0,0,\dots,0)$, the standard coding of $t$ is just the expansion of $t$ in base $m$, except that we name the digits $1,\dots,m$ rather than $0,\dots,m-1$.

For $i_1,\dots,i_n\in\II$, let $I_{i_1,\dots,i_n}:=\phi_{i_1}\circ\dots\circ\phi_{i_n}([0,1])$. We will call $(i_1,\dots,i_n)$ the coding of $I_{i_1,\dots,i_n}$. For $n\in\NN$ and $t\in[0,1]$, let $I_n(t)$ denote the unique interval $I_{i_1,\dots,i_n}$ that contains $t$ and such that the standard coding of $t$ begins with $(i_1,\dots,i_n)$. For fixed $n$, we enumerate the intervals $I_{i_1,\dots,i_n}$ from left to right as $I_{n,j}: j=1,\dots,m^n$.

Let $\mathcal{T}_0$ denote the set of endpoints of the intervals $I_{n,j}$ ($n\in\NN$, $j=1,\dots,m^n$). These are the points that have two distinct codings.

Fix $t\sim (i_1,i_2,\dots) \in[0,1]$, and for $n\in\NN$, let $u_n$ and $v_n$ denote the left and right endpoints, respectively, of $I_n(t)$.
Thus, $u_n\leq t\leq v_n$, and $v_n-u_n=|I_n(t)|=c_{i_1}\cdots c_{i_n}$. Furthermore, let
\begin{gather*}
k_i(n;t):=\#\{j\leq n: i_j=i\},\\
d_i(n;t):=\frac{k_i(n;t)}{n},
\end{gather*}
for $n\in\NN$ and $i=0,1,\dots,m-1$, and define
\begin{equation*}
d_i(t):=\lim_{n\to\infty} d_i(n;t),
\end{equation*}
provided the limit exists. Thus, $d_i(t)$ is the frequency of the ``digit" $i$ in the standard coding of $t$. Since
\begin{equation*}
f(v_n)-f(u_n)=\pm\big(S_{i_1}\circ \dots \circ S_{i_n}(\mathbf{b})-S_{i_1}\circ \dots \circ S_{i_n}(\mathbf{a})\big),
\end{equation*}
we have
\begin{equation*}
|f(v_n)-f(u_n)|=|S_{i_1}\circ\dots\circ S_{i_n}(\mathbf{b}-\mathbf{a})|=\lambda_{i_1}\cdots \lambda_{i_n}.
\end{equation*}
This gives, for $\alpha>0$, the useful expression
\begin{equation}
\frac{|f(v_n)-f(u_n)|}{(v_n-u_n)^\alpha}=\prod_{i=1}^{m}\left(\frac{\lambda_i}{c_i^\alpha}\right)^{k_i(n;t)}
=\left(\prod_{i=1}^{m}\left(\frac{\lambda_i}{c_i^\alpha}\right)^{d_i(n;t)}\right)^n,
\label{eq:dq-product-expression}
\end{equation}
which (crucially!) does not depend on the signature $\boldsymbol{\eps}$.

An important tool in this paper is the following generalization of Eggleston's theorem \cite{Eggleston}, due to Li and Dekking (see \cite{Li-Dekking}, Theorem 1 and eq. (35) on p.~198):
\begin{equation}
\dim_H\{t\in(0,1): d_i(t)=p_i, i=1,\dots,m\}=H(p_1,\dots,p_m),
\label{eq:general-Eggleston}
\end{equation}
where $H$ was defined in \eqref{eq:h-definition}. Generalizing Eggleston's theorem in a different direction, Barreira et al.~\cite{BSS} proved, for the special case when $c_i=1/m$ for each $i$ and $\boldsymbol\eps=(0,0,\dots,0)$, that
\begin{align}
\begin{split}
\dim_H&\left\{t: \lim_{n\to\infty}\sum_{i=1}^{m}\beta_i d_i(n;t)=\alpha\right\}\\
&=\max\left\{H(\mathbf{p}): \mathbf{p}=(p_1,\dots,p_{m})\in\Delta_m, \sum_{i=1}^{m}\beta_i p_i=\alpha\right\}
\end{split}
\label{eq:Barreira-result}
\end{align}
for real numbers $\alpha,\beta_1,\dots,\beta_{m}$. In Section \ref{sec:Holder-exponent-calculation}, we will develop an expression for the nondirectional H\"older exponent $\tilde{\alpha}_f(t)$ which similarly involves a linear combination of the partial densities $d_i(n;t)$. But there is also a correction term, necessary to deal with points $t$ with exceptionally long strings of $1$'s or $m$'s in their codings. Thus, we will need a further extension of \eqref{eq:Barreira-result}, proved in Proposition \ref{prop:Holder-halves}.

We end this section with a proof of Proposition \ref{prop:duality}.



\begin{proof}[Proof of Proposition \ref{prop:duality}]
Let $\eta(\alpha)$ denote the expression on the right hand side of \eqref{eq:duality}. Assume initially that $\alpha\in(\alpha_{\min},\alpha_{\max})$.
We first show that $\beta^*(\alpha)\leq \eta(\alpha)$. Since $\lim_{q\to\pm\infty}(\alpha q+\beta(q))=\infty$, there is a (unique) value $q^*$ of $q$ that minimizes $\alpha q+\beta(q)$. Differentiating implicitly in \eqref{eq:scaling-equation} and setting $\beta'(q^*)=-\alpha$ yields
\begin{equation}
\sum_{i\in\II_+}\lambda_i^{q^*}c_i^{\beta(q^*)}(\log\lambda_i-\alpha\log c_i)=0.
\label{eq:q-min-equation}
\end{equation}
Set $p_i=\lambda_i^{q^*}c_i^{\beta(q^*)}$ for $i\in\II_+$, and $p_i=0$ for $i\in\II_0$. Then by \eqref{eq:q-min-equation}, $\mathbf{p}=(p_1,\dots,p_m)$ satisfies the constraints in \eqref{eq:duality}, and
\begin{align*}
\sum_{i=1}^m p_i\log p_i&=\sum_{i\in\II_+} \lambda_i^{q^*}c_i^{\beta(q^*)}\left\{q^*\log\lambda_i+\beta(q^*)\log c_i\right\}\\
&=\beta^*(\alpha)\sum_{i\in\II_+}\lambda_i^{q^*}c_i^{\beta(q^*)}\log c_i=\beta^*(\alpha)\sum_{i=1}^m p_i\log c_i.
\end{align*}
Hence, $\beta^*(\alpha)=H(p_1,\dots,p_m)\leq \eta(\alpha)$.

Conversely, let $\mathbf{p}=(p_1,\dots,p_m)\in\Delta_m^0$ such that $\sum_{i\in\II_+}p_i(\log\lambda_i-\alpha\log c_i)=0$; we must show that $H(\mathbf{p})\leq \alpha q+\beta(q)$ for each $q\in\RR$. By continuity of $H$, it is enough to show this when $p_i>0$ for each $i\in\II_+$. Since $\beta(q)$ is decreasing in $q$, we need to show in view of \eqref{eq:scaling-equation} that
\begin{equation}
\sum_{i\in\II_+}\lambda_i^q c_i^{H(\mathbf{p})-\alpha q}\geq 1.
\label{eq:sum-lower-bound}
\end{equation}
Using the concavity of $\log x$, we have (with all summations over $i\in\II_+$)
\begin{align*}
\log \sum\lambda_i^q c_i^{H(\mathbf{p})-\alpha q} &= \log \sum p_i\left(\frac{\lambda_i^q}{p_i}c_i^{H(\mathbf{p})-\alpha q}\right)\\
&\geq \sum p_i\log\left(\frac{\lambda_i^q}{p_i}c_i^{H(\mathbf{p})-\alpha q}\right)\\
&=\sum p_i\left\{q\log\lambda_i-\log p_i+\big(H(\mathbf{p})-\alpha q\big)\log c_i\right\}\\
&=q\sum p_i(\log\lambda_i-\alpha\log c_i)+\sum p_i\left\{H(\mathbf{p})\log c_i-\log p_i\right\}\\
&=0,
\end{align*}
since the last summation vanishes by definition of $H(\mathbf{p})$. Exponentiating gives \eqref{eq:sum-lower-bound}.
Thus, $\beta^*(\alpha)\geq \eta(\alpha)$.

For $\alpha\in\{\alpha_{\min},\alpha_{\max}\}$, \eqref{eq:duality} now follows from the continuity of $\beta^*(\alpha)$ and $\eta(\alpha)$ in $[\alpha_{\min},\alpha_{\max}]$. The former is well known; the latter is a consequence of the continuity of $H(\mathbf{p})$ with respect to $\mathbf{p}$ and the continuity of $\sum_{i\in\mathcal{I_+}}p_i(\log\lambda_i-\alpha\log c_i)$ with respect to $\alpha$ and $\mathbf{p}$.
\end{proof}

\section{Proofs of Proposition \ref{prop:zero-or-none} and Theorem \ref{thm:differentiability}} \label{sec:differentiability-proofs}

In this and later sections, let
\begin{equation*}
c_{\min}:=\min_{i\in\II}c_i, \qquad c_{\max}:=\max_{i\in\II}c_i.
\end{equation*}
We begin with a useful lemma, whose easy proof is left to the reader.

\begin{lemma} \label{lem:generalized-chord-lemma}
Let $t\in(0,1)$, and suppose $f'(t)$ exists and is finite. If $(s_n)_n$ and $(t_n)_n$ are any two sequences converging to $t$ such that $(t_n-t)/(t_n-s_n)$ is bounded, then
\begin{equation*}
\frac{f(t_n)-f(s_n)}{t_n-s_n}\to f'(t).
\end{equation*}
\end{lemma}

\begin{proof}[Proof of Proposition \ref{prop:zero-or-none}]
Assume $f'(t)$ exists but $f'(t)\neq 0$. Then, by Lemma \ref{lem:generalized-chord-lemma},
\begin{equation*}
\frac{|f(v_n)-f(u_n)|}{v_n-u_n}\to |f'(t)|>0.
\end{equation*}
Since $v_{n+1}-u_{n+1}=c_{i_{n+1}}(v_n-u_n)$ and $|f(v_{n+1})-f(u_{n+1})|=\lambda_{i_{n+1}}|f(v_n)-f(u_n)|$, it follows that
\begin{equation*}
\frac{\lambda_{i_{n+1}}}{c_{i_{n+1}}}\to 1.
\end{equation*}
This is possible only if $\lambda_i=c_i$ for some $i$, and then $\lambda_{i_n}=c_{i_n}$ for all sufficiently large $n$. Suppose this is the case. Fix $k\in\II$ such that $\lambda_k\neq c_{k}$. For each $n$, let $s_n$ and $t_n$ be the left and right endpoints, respectively, of the interval $I_{i_1,\dots,i_{n-1},k}$. Then
\begin{equation*}
\frac{t_n-t}{t_n-s_n}\leq \frac{v_{n-1}-u_{n-1}}{t_n-s_n}=\frac{1}{c_{k}},
\end{equation*}
so Lemma \ref{lem:generalized-chord-lemma} implies
\begin{equation*}
\frac{|f(t_n)-f(s_n)|}{t_n-s_n}\to |f'(t)|.
\end{equation*}
But this is impossible, since
\begin{equation*}
\frac{|f(t_n)-f(s_n)|}{t_n-s_n}=\frac{|f(v_n)-f(u_n)|}{v_n-u_n} \cdot \frac{\lambda_{k}c_{i_n}}{\lambda_{i_n}c_{k}},
\end{equation*}
and, for all large enough $n$, the last fraction on the right is constant $\lambda_k/c_k\neq 1$.
\end{proof}

The following lemma is a direct generalization of \cite[Lemma 3]{Lax}.

\begin{lemma} \label{lem:Lax-general}
Let $M_n:=\inf\{j>n:i_j=i_n\}$ for $n\in\NN$. If $d_i(t)$ exists and $d_i(t)>0$ for each $i\in\{1,\dots,m\}$, then $M_n=n+o(n)$.
\end{lemma}

\begin{proof}
Suppose $i_n=1$. Then $k_1(M_n;t)=k_1(n;t)+1$, so
\begin{equation*}
\frac{k_1(M_n;t)}{M_n}\cdot\frac{M_n}{n}=\frac{k_1(n;t)}{n}+\frac{1}{n}.
\end{equation*}
Since $d_1(t)$ exists and is strictly positive, it follows that $M_n/n\to 1$ along the subsequence $\{n:i_n=1\}$. Similarly considering the other digits yields $M_n=n+o(n)$.
\end{proof}

\begin{lemma} \label{lem:when-derivative-zero}
Suppose that $d_i(t)>0$ for $i=1,\dots,m$, and that
\begin{equation}
\sum_{i=1}^{m}d_i(t)\log(\lambda_i/c_i)<0.
\label{eq:sufficient-density-inequality}
\end{equation}
Then $f'(t)=0$.
\end{lemma}

\begin{proof}
We give a short proof, based on Lax's argument \cite{Lax}, for the signature $\boldsymbol{\eps}=(0,0,\dots,0)$. For the general case, the lemma will follow from Corollary \ref{cor:holder-exponent-densities}.

Let $t'\neq t$, and let $n$ be the largest integer such that $t'\in I_n(t)$. Then
\begin{equation*}
|f(t')-f(t)|\leq K|f(v_n)-f(u_n)|,
\end{equation*}
where $K:=2\max_{0\leq t\leq 1}|f(t)|$; and
$|t'-t|\geq c_{\min}^{M_{n+1}-n}c_{i_1}\cdots c_{i_n}$,
with $M_n$ defined as in Lemma \ref{lem:Lax-general}. Thus, using \eqref{eq:dq-product-expression},
\begin{equation*}
\frac{|f(t')-f(t)|}{|t'-t|}\leq K c_{\min}^{-(M_{n+1}-n)}\left[\prod_{i=1}^{m}\left(\frac{\lambda_i}{c_i}\right)^{d_i(n;t)}\right]^n\to 0,
\end{equation*}
since $\prod(\lambda_i/c_i)^{d_i(n;t)}\to \prod(\lambda_i/c_i)^{d_i(t)}<1$ by \eqref{eq:sufficient-density-inequality}, and $c_{\min}^{-(M_{n+1}-n)}$ grows subexponentially by Lemma \ref{lem:Lax-general}. 
\end{proof}

\begin{proof}[Proof of Theorem \ref{thm:differentiability}]
Statement (i) follows immediately from Proposition \ref{prop:zero-or-none} and Lemma \ref{lem:generalized-chord-lemma}, since
\begin{equation*}
\frac{|f(v_n)-f(u_n)|}{v_n-u_n}=\frac{\lambda_{i_1}\cdots \lambda_{i_n}}{c_{i_1}\cdots c_{i_n}}\geq 1.
\end{equation*}

To prove (ii), we assume first that $\sum_{i=1}^{m}c_i\log(\lambda_i/c_i)>0$. By the strong law of large numbers, $d_i(n;t)\to c_i$ for almost every $t$ and $i=1,\dots,m$, and so
\begin{equation*}
\lim_{n\to\infty}\prod_{i=1}^{m}\left(\frac{\lambda_i}{c_i}\right)^{d_i(n;t)}=\prod_{i=1}^{m}\left(\frac{\lambda_i}{c_i}\right)^{c_i}>1
\end{equation*}
for almost every $t$. Thus, \eqref{eq:dq-product-expression} gives
\begin{equation*}
\frac{|f(v_n)-f(u_n)|}{v_n-u_n}\to \infty
\end{equation*}
for almost all $t$, and hence, $f$ is differentiable almost nowhere.

The case when $\sum_{i=1}^{m}c_i\log(\lambda_i/c_i)=0$ needs a separate argument. In this case, we view the numbers $k_i(n;t)$ as random variables on the Lebesgue probability space $[0,1]$ with the Lebesgue (or Borel) $\sigma$-algebra and Lebesgue measure. Since the ``digits" $i_1,i_2,\dots$ in the coding of $t$ are independent and identically distributed, the sums
\begin{equation*}
\sum_{i=1}^{m}k_i(n;t)\log(\lambda_i/c_i), \qquad n=0,1,2,\dots
\end{equation*}
follow a random walk with steps chosen randomly from the set $\{\log(\lambda_i/c_i): i=1,\dots,m\}$, in which the expected step size is $\sum_{i=1}^{m}c_i\log(\lambda_i/c_i)=0$. Then, for example, the law of the iterated logarithm implies that for almost all $t$,
\begin{equation*}
\sum_{i=1}^{m}k_i(n;t)\log(\lambda_i/c_i)>0 \qquad\mbox{for infinitely many $n$}.
\end{equation*}
Exponentiating and using \eqref{eq:dq-product-expression}, it follows that $f$ is differentiable almost nowhere.

The claim that $f'(t)=0$ at uncountably many $t$ if $\lambda_i<c_i$ for some $i$ will follow once we prove Theorem \ref{thm:Hausdorff-dimension}. 

Finally, the first part of (iii) follows from Lemma \ref{lem:when-derivative-zero}, since $d_i(t)=c_i$ for $i=1,\dots,m$ and almost every $t\in(0,1)$, so the hypothesis of (iii) implies \eqref{eq:sufficient-density-inequality} for almost all $t$. The second part of (iii) will follow once we prove Theorem \ref{thm:Hausdorff-dimension}.
\end{proof}

\section{Proof of Theorem \ref{thm:simple-Holder-spectrum}} \label{sec:simple-Holder}

In this section we assume the hypotheses of Theorem \ref{thm:simple-Holder-spectrum}: $c_i=1/m$ for each $i$, and $\boldsymbol{\eps}=(0,0,\dots,0)$. Note that $\mathcal{T}_0$ is then simply the set of all $m$-adic rational numbers in $[0,1]$.

\begin{proof}[Proof of Theorem \ref{thm:simple-Holder-spectrum}]
Fix $t_0\in[0,1]$, and assume $f\in C^\alpha(t_0)$ with $\alpha>1$. Let $N$ be the greatest integer strictly less than $\alpha$. Thus, there are polynomials $P_1,\dots,P_d$ of degree at most $N$ and a constant $C>0$ such that
\begin{equation}
\big|\big(f_1(t)-P_1(t-t_0),\dots,f_d(t)-P_d(t-t_0)\big)\big|\leq C|t-t_0|^\alpha \quad \mbox{for all $t\in[0,1]$},
\label{eq:multi-dim-Holder-property}
\end{equation}
where we write $f=(f_1,\dots,f_d)$. We need to show that $P_i(t)\equiv f_i(t_0)$ for $i=1,\dots,d$. Aiming for a contradiction, assume that this is not the case. Write
\begin{equation*}
P_i(t)=f_i(t_0)+a_{i,1}t+a_{i,2}t^2+\dots+a_{i,N}t^{N}, \qquad i=1,\dots,d.
\end{equation*}
For each $i$, set $l_i=\infty$ if $P_i$ is constant, and otherwise, set $l_i:=\min\{j\geq 1:a_{i,j}\neq 0\}$. Note that at least one $l_i$ is finite; let $l:=\min_{1\leq i\leq d}l_i$. We can divide by $(t-t_0)^l$ in \eqref{eq:multi-dim-Holder-property} to obtain
\begin{equation}
\lim_{t\to t_0}\frac{f(t)-f(t_0)}{(t-t_0)^l}=(a_{1,l},\dots,a_{d,l})=:\mathbf{x}\neq 0,
\label{eq:l-derivative}
\end{equation}
since
\begin{equation*}
\left|\frac{f_i(t)-f_i(t_0)}{(t-t_0)^l}-a_{i,l}-a_{i,l+1}(t-t_0)-\dots-a_{i,N}(t-t_0)^{N-l}\right|\leq C|t-t_0|^{\alpha-l}\to 0
\end{equation*}
for $i=1,\dots,d$. Observe that \eqref{eq:l-derivative} is a rather strong statement. For instance, if $l=1$ it says that $f$ has a well-defined nonzero derivative at $t_0$, which is impossible in view of Proposition \ref{prop:zero-or-none}. Thus, we must have $l\geq 2$.

\bigskip
{\em Case 1.} Assume first that $t_0\in\mathcal{T}_0$, say $t_0=k/m^n$. It will be sufficient to consider $t>t_0$. Since the graph of $f$ on the interval $[k/m^n,(k+1)/m^n]$ is an affine copy of the full graph of $f$, we can and do assume without loss of generality that $t_0=0$. For each $i\in\NN$, \eqref{eq:l-derivative} gives
\begin{equation*}
\lim_{n\to\infty}\frac{f(i/m^n)}{(i/m^n)^l}=\mathbf{x},
\end{equation*}
so that
\begin{align*}
m^{nl}\left[f\left(\frac{i+1}{m^n}\right)-f\left(\frac{i}{m^n}\right)\right]
&=\displaystyle\frac{\displaystyle f\left(\frac{i+1}{m^n}\right)}{\left(\displaystyle\frac{i+1}{m^n}\right)^l}\cdot(i+1)^l-\frac{\displaystyle f\left(\frac{i}{m^n}\right)}{\left(\displaystyle\frac{i}{m^n}\right)^l}\cdot i^l\\
&\to \{(i+1)^l-i^l\}\mathbf{x}.
\end{align*}
Letting
\begin{equation*}
\delta_{n,i}:=\left|f\left(\frac{i+1}{m^n}\right)-f\left(\frac{i}{m^n}\right)\right|,
\end{equation*}
for $n\in\NN$ and $i=0,1,\dots,m^n-1$, it follows that
\begin{equation}
m^{nl}\delta_{n,i}\to \{(i+1)^l-i^l\}|\mathbf{x}|, \qquad i\in\NN.
\label{eq:m-adic-difference-expression}
\end{equation}

On the other hand, for $i<m$ it follows from \eqref{eq:our-self-affine-functions-intro} that $\delta_{n,i}=\lambda_1^{n-1}\lambda_{i+1}$.
In particular, setting $i=1$ in \eqref{eq:m-adic-difference-expression} gives
\begin{equation*}
m^{nl}\lambda_1^{n-1}\lambda_2\to (2^l-1)|\mathbf{x}|.
\end{equation*}
But similarly we have $\delta_{n,m}=\lambda_1^{n-1}\lambda_2$, 
so setting $i=m$ in \eqref{eq:m-adic-difference-expression} we obtain
\begin{equation*}
m^{nl}\lambda_1^{n-1}\lambda_2\to \{(m+1)^l-m^l\}|\mathbf{x}|.
\end{equation*}
Thus, $(m+1)^l-m^l=2^l-1$. But this is impossible, since $m\geq 2$ and the function $x\mapsto x^l$ is strictly convex on $(0,\infty)$ for $l\geq 2$.

\bigskip
{\em Case 2.} Assume now that $t_0\not\in\mathcal{T}_0$. We initially assume also that $\lambda_i>0$ for each $i$. Note that $\delta_{n,k}$ is a product of some combination of the $\lambda_i$, and is therefore nonzero. Define the set
\begin{equation*}
\mathcal{R}:=\left\{\frac{\delta_{n,k-1}}{\delta_{n,k}}\,\bigg|\, n\in\NN,\ k=1,2,\dots,m^n\right\}.
\end{equation*}
It is easy to see that
\begin{equation*}
\mathcal{R}=\left\{\frac{\lambda_{i-1}}{\lambda_i}\left(\frac{\lambda_{m}}{\lambda_1}\right)^{n-1}\,\bigg|\, n\in\NN,\ i=2,\dots,m \right\},
\end{equation*}
and therefore $\mathcal{R}$ has no limit points in $(0,\infty)$.

Define the map $T:[0,1)\to[0,1)$ by $T(t):=mt \mod 1$, and denote by $T^n$ the $n$th iterate of $T$. Since $t_0$ is not $m$-adic rational, there is a number $\tau\in(0,1)$ and a subsequence $(n_\nu)$ of $\NN$ such that $T^{n_\nu}(t_0)\to \tau$. To avoid notational clutter, we shall for the remainder of the proof suppress the index $\nu$ and simply write $n$ instead of $n_\nu$. By continuity of $f$, 
\begin{equation*}
f(T^n(t_0))\to f(\tau).
\end{equation*}
Here and in what follows, convergence takes place along the subsequence $(n_\nu)$ as $\nu\to\infty$.
Assume initially that $f(\tau)\neq 0$. Write $I_n(t_0)=[u_n,v_n]$. Then $m^n(t_0-u_n)=T^n(t_0)$, so
\begin{equation*}
\frac{|f(t_0)-f(u_n)|}{(t_0-u_n)^l}=\frac{m^{nl}\lambda_{i_1}\cdots\lambda_{i_n}}{(T^n(t_0))^l}|f(T^n(t_0))|.
\end{equation*}
Thus, \eqref{eq:l-derivative} implies
\begin{equation}
m^{nl}\lambda_{i_1}\cdots\lambda_{i_n}\to \frac{\tau^l|\mathbf{x}|}{|f(\tau)|}>0.
\label{eq:own-interval-limit}
\end{equation}
Next, for $j=1,2,\dots$ we have
\begin{align*}
m^{nl}\big[f(u_n&-(j-1)m^{-n})-f(u_n-jm^{-n})\big]\\
&=(-1)^l\left[(T^n(t_0)+j-1)^l\cdot\frac{f(u_n-(j-1)m^{-n})-f(t_0)}{(u_n-(j-1)m^{-n}-t_0)^l}\right.\\
&\qquad\qquad\qquad\qquad\qquad \left.-(T^n(t_0)+j)^l\cdot\frac{f(u_n-jm^{-n})-f(t_0)}{(u_n-jm^{-n}-t_0)^l}\right]\\
&\to (-1)^{l-1}\big\{(\tau+j)^l-(\tau+j-1)^l\big\}\mathbf{x},
\end{align*}
and so
\begin{equation}
m^{nl}\big|f(u_n-(j-1)m^{-n})-f(u_n-jm^{-n})\big| \to \big\{(\tau+j)^l-(\tau+j-1)^l\big\}|\mathbf{x}|.
\label{eq:intervals-on-left-limit}
\end{equation}
Write $u_n=k_n/m^n$. Then
\begin{equation}
\big|f(u_n-(j-1)m^{-n})-f(u_n-jm^{-n})\big|=\lambda_{i_1}\cdots\lambda_{i_n}b_{n,1}\cdots b_{n,j},
\label{eq:intervals-on-left-norm}
\end{equation}
where
\begin{equation*}
b_{n,i}:=\frac{\delta_{n,k_n-i}}{\delta_{n,k_n-i+1}}, \qquad i=1,\dots,k_n.
\end{equation*}
Combining \eqref{eq:own-interval-limit}, \eqref{eq:intervals-on-left-limit} and \eqref{eq:intervals-on-left-norm}, we obtain
\begin{equation*}
b_{n,1}\cdots b_{n,j} \to \frac{(\tau+j)^l-(\tau+j-1)^l}{\tau^l} |f(\tau)|, \qquad j=1,2,\dots
\end{equation*}
and so 
\begin{equation*}
b_{n,j}=\frac{b_{n,1}\cdots b_{n,j}}{b_{n,1}\cdots b_{n,j-1}}\to \frac{(\tau+j)^l-(\tau+j-1)^l}{(\tau+j-1)^l-(\tau+j-2)^l}>1, \qquad j=2,3,\dots,
\end{equation*}
since $l\geq 2$ and the function $x\mapsto x^l$ is strictly convex on $(0,\infty)$. However, $b_{n,j}\in\mathcal{R}$, and since $\mathcal{R}$ does not have a limit point in $(0,\infty)$, it follows that (for each fixed $j$) $b_{n,j}$ is eventually constant, say 
\begin{equation*}
b_{n,j}\equiv b_j:=\frac{(\tau+j)^l-(\tau+j-1)^l}{(\tau+j-1)^l-(\tau+j-2)^l}, \qquad j=2,3,\dots. 
\end{equation*}
But then $b_j\in\mathcal{R}$, $b_j\neq 1$ and $\lim_{j\to\infty}b_j=1$, contradicting again that $\mathcal{R}$ does not have a limit point in $(0,\infty)$.

If $f(\tau)=0$, then instead of \eqref{eq:own-interval-limit} we have
$m^{nl}\lambda_{i_1}\cdots\lambda_{i_n}\to \infty$,
and so
\begin{equation*}
\frac{|f(v_n)-f(t_0)|}{(v_n-t_0)^l}=m^{nl}\lambda_{i_1}\cdots\lambda_{i_n} \frac{|\mathbf{b}-f(T^n(t_0))|}{(1-T^n(t_0))^l}\to\infty,
\end{equation*}
since $T^n(t_0)\to \tau\in(0,1)$ and $|\mathbf{b}-f(T^n(t_0))|\to |\mathbf{b}-f(\tau)|=|\mathbf{b}|=1$. But this contradicts \eqref{eq:l-derivative}.

It remains only to deal with the case when $\lambda_i=0$ for some $i$. In this case, there is for each $n\in\NN$ a number $j_n\leq m$ such that
\begin{equation*}
f(u_n-(j_n-1)m^{-n})-f(u_n-j_n m^{-n})=0.
\end{equation*}
We can then find a number $j\leq m$ such that $j_n=j$ for infinitely many $n$. But for this $j$, \eqref{eq:intervals-on-left-limit} is impossible, and we once again have a contradiction to \eqref{eq:l-derivative}. 
\end{proof}

\section{Calculation of $\tilde{\alpha}_f(t)$} \label{sec:Holder-exponent-calculation}

In this section we derive a precise (but somewhat technical) expression for $\tilde{\alpha}_f(t)$ in terms of the coding of $t$. Assume without loss of generality that 
\begin{equation}
\frac{\log\lambda_m}{\log c_m}\geq \frac{\log\lambda_1}{\log c_1}.
\label{eq:ratio-assumption}
\end{equation}
(If this does not hold, simply switch the roles of the digits $1$ and $m$ everywhere in what follows.) Let
\begin{equation}
K:=\begin{cases}
0 & \mbox{if $\lambda_m=0$},\\
\left(\frac{\log c_m}{\log c_1}\right)\log\lambda_1-\log\lambda_m & \mbox{if $\lambda_m>0$}.
\end{cases}
\label{eq:definition-of-K}
\end{equation}
Note that $K\geq 0$. For $t\sim (i_1,i_2,\dots)$, define the ``look back" run length
\begin{equation}
L_n(t):=\max\{j\leq n: i_{n-j+1}=\dots=i_{n-1}=i_n=m\},
\label{eq:run-length}
\end{equation}
or $L_n(t):=0$ if $i_n<m$. For $i\in\II$, let
\begin{equation*}
i_{\to}:=\begin{cases}
i+1 & \mbox{if $\eps_i=0$},\\
i-1 & \mbox{if $\eps_i=1$}.
\end{cases}
\end{equation*}
There are three essentially different cases to consider. We deal with the case $t\in\mathcal{T}_0$ separately, in Theorem \ref{thm:pointwise-Holder-at-endpoints}. If $t\not\in\KK_\phi$, then $f$ is constant on $I_n(t)$ for some $n$, and $\tilde{\alpha}_f(t)=\infty$. The critical case, addressed in Theorem \ref{thm:individual-Holder-exponents} below, is when $t\not\in\mathcal{T}_0$ and $t\in\KK_\phi$. We make the convention that $\log 0:=-\infty$ and $0\log 0:=0$.

\begin{theorem} \label{thm:individual-Holder-exponents}
Assume \eqref{eq:ratio-assumption}, and let $t\sim(i_1,i_2,\dots)\in\KK_\phi\backslash \mathcal{T}_0$.
\begin{enumerate}[(i)]
\item There is a unique number $\alpha_0:=\alpha_0(t)\in[\alpha_{\min},\alpha_{\max}]$ such that
\begin{equation}
\limsup_{n\to\infty}\left[\sum_{i=1}^m d_i(n;t)(\log\lambda_i-\alpha_0\log c_i)\right]=0,
\end{equation}
and if $\max\{\eps_1,\eps_m\}=1$, then $\tilde{\alpha}_f(t)=\alpha_0(t)$.
\item Suppose $\eps_1=\eps_m=0$. Let
\begin{equation*}
\chi_n(t):=\begin{cases}
1 & \mbox{if $\eps_i=\eps_{i_{\to}}$ and $i_{\to}\in\II_+$},\\
0 & \mbox{otherwise},
\end{cases}
\end{equation*}
where $i:=i_{n-L_n(t)}$. Then there is a unique number $\alpha_1:=\alpha_1(t)\in[\alpha_{\min},\alpha_{\max}]$ such that
\begin{equation}
\limsup_{n\to\infty}\left[\sum_{i=1}^m d_i(n;t)(\log\lambda_i-\alpha_1\log c_i)+K\chi_n(t)\frac{L_n(t)}{n}\right]=0.
\label{eq:general-Holder-limsup-equation}
\end{equation}
Moreover, $\alpha_1(t)\leq\alpha_0(t)$, and $\tilde{\alpha}_f(t)=\alpha_1(t)$.
\end{enumerate}
\end{theorem}

\begin{corollary} \label{cor:holder-exponent-densities}
Assume \eqref{eq:ratio-assumption}. Let $t\sim(i_1,i_2,\dots)\in\KK_\phi\backslash\mathcal{T}_0$. If $d_i(t)$ exists for $i=1,\dots,m$ and $d_m(t)<1$, then
\begin{equation}
\tilde{\alpha}_f(t)=\frac{\sum_{i=1}^{m}d_i(t)\log\lambda_i}{\sum_{i=1}^{m}d_i(t)\log c_i}.
\label{eq:Holder-exponent-if-densities-exist}
\end{equation}
\end{corollary}

\begin{proof}
Combining the digits $1,\dots,m-1$ into a single digit ``other", the proof of Lemma \ref{lem:Lax-general} shows that $L_n(t)=o(n)$. Hence, $\alpha_0(t)=\alpha_1(t)=$ the right hand side of \eqref{eq:Holder-exponent-if-densities-exist}.
\end{proof}

Corollary \ref{cor:holder-exponent-densities} implies Lemma \ref{lem:when-derivative-zero}: Under the hypotheses of that lemma, \eqref{eq:Holder-exponent-if-densities-exist} gives $\tilde{\alpha}_f(t)>1$, and therefore $f'(t)=0$.



The proof of Theorem \ref{thm:individual-Holder-exponents} uses the following lemmas.

\begin{lemma} \label{lem:existence-of-alphas}
Under the respective hypotheses of Theorem \ref{thm:individual-Holder-exponents}, the numbers $\alpha_0(t)$ and $\alpha_1(t)$ exist and are unique, and lie in $[\alpha_{\min},\alpha_{\max}]$.
\end{lemma}

\begin{proof}
We demonstrate the existence and uniqueness of $\alpha_1(t)$ in case (ii); the proof concerning $\alpha_0(t)$ is more straightforward. 

Assume $\eps_1=\eps_m=0$. We first show that $\chi_n(t)$ is well defined. Since $\eps_m=0$, there is at least one index $k$ such that $i_k<m$, as otherwise we would have $t=1\in\mathcal{T}_0$. We may then assume (since we are going to let $n\to\infty$) that $n>k$ with $k$ chosen as above, so that $n-L_n(t)>0$. By the definition of $L_n(t)$ in \eqref{eq:run-length}, $i:=i_{n-L_n(t)}<m$ and so $i+1\in\II$. If furthermore $i>1$, then $i-1\in\II$ as well and so $i_{\to}$ is well defined. If $i=1$ (in which case $i-1\not\in\II$), then $\eps_i=\eps_1=0$ and so $i_{\to}=i+1=2$. Thus, in all cases $i_{\to}$ is well defined and as a result, $\chi_n(t)$ is well defined.

Writing the expression in square brackets in \eqref{eq:general-Holder-limsup-equation} as $\sigma_n \alpha_1+\tau_n$, we have $0<-\log c_{\max}\leq\sigma_n\leq -\log c_{\min}$ and $\min_{i\in\II_+}\log\lambda_i\leq\tau_n\leq\max_{i\in\II_+}\log\lambda_i+K$. Thus for each $\alpha>0$, $\vartheta(\alpha):=\limsup(\sigma_n\alpha+\tau_n)$ exists and moreover, $\vartheta(\alpha)$ is strictly increasing and continuous in $\alpha$, since the $\limsup$ of a sequence of linear functions is convex, and hence, continuous. Furthermore, we claim that
\begin{equation}
\vartheta(\alpha_{\min})\leq 0\leq \vartheta(\alpha_{\max}).
\label{eq:from-negative-to-positive}
\end{equation}
This is clear when $K=0$. Otherwise, by \eqref{eq:definition-of-K}, $K=(\log c_m/\log c_1)\log\lambda_1-\log\lambda_m$, and since $\log\lambda_i-\alpha_{\min}\log c_i\leq 0$ for $i=1,\dots,m-1$ and $\alpha_{\min}\leq \log\lambda_1/\log c_1$, we have for each $n\in\NN$,
\begin{align*}
\sum_{i=1}^m k_i(n;t)&(\log\lambda_i-\alpha_{\min}\log c_i)+K\chi_n(t)L_n(t)\\
&\leq k_m(n;t)\left(\log\lambda_m-\frac{\log\lambda_1}{\log c_1}\log c_m\right)+K\chi_n(t)L_n(t)\\
&=K\left\{\chi_n(t)L_n(t)-k_m(n;t)\right\}\leq 0.
\end{align*}
Dividing by $n$ and taking $\limsup$ gives the first inequality in \eqref{eq:from-negative-to-positive}. The second inequality is clear. Therefore, $\vartheta(\alpha)$ has a unique zero in $[\alpha_{\min},\alpha_{\max}]$. 
\end{proof}

\begin{lemma} \label{lem:boundedness}
Let $I_{n,j}$ and $I_{n,j+1}$ have codings $(i_1,\dots,i_n)$ and $(i_1',\dots,i_n')$, respectively. Let $n_0:=\min\{\nu:i_\nu\neq i_\nu'\}$. Let $k_i:=\#\{\nu\leq n: i_\nu=i\}$ and $k_i':=\#\{\nu\leq n: i_\nu'=i\}$. If either $\max\{\eps_1,\eps_m\}=1$ or $\eps_{i_{n_0}}\neq \eps_{i_{n_0}'}$, then $|k_i-k_i'|\leq 2$ for $i=1,\dots,m$.
\end{lemma}

\begin{proof}
We have $(i_1,\dots,i_{n_0-1})=(i_1',\dots,i_{n_0-1}')$ and $|i_{n_0}-i_{n_0}'|=1$. If, say, $\eps_1=0$ and $\eps_m=1$, then $i_{n_0+1}$ and $i_{n_0+1}'$ can each be either $1$ or $m$ (depending on the values of $\eps_{i_{n_0}}$ and $\eps_{i_{n_0}'}$), but $(i_{n_0+2},\dots,i_n)=(i_{n_0+2}',\dots,i_n')=(1,\dots,1)$ in all cases. A similar conclusion holds if $\eps_1=1$ and $\eps_m=0$, with $(1,\dots,1)$ replaced by $(m,\dots,m)$.
If $\eps_1=\eps_m=1$, then $(i_{n_0+1},\dots,i_n)$ and $(i_{n_0+1}',\dots,i_n')$ consist of alternating $1$'s and $m$'s. (For instance, if $\eps_{i_{n_0}}=\eps_{i_{n_0}'}=0$ and $n-n_0$ is odd, then $(i_{n_0+1},\dots,i_n)=(m,1,\dots,m,1,m)$ and $(i_{n_0+1}',\dots,i_n')=(1,m,\dots,1,m,1)$ or vice versa.) Finally, if $\eps_1=\eps_m=0$ and $\eps_{i_{n_0}}\neq \eps_{i_{n_0}'}$, then $(i_{n_0+1},\dots,i_n)=(i_{n_0+1}',\dots,i_n')=(1,\dots,1)$ or $(m,\dots,m)$. Thus, in all cases, $|k_i-k_i'|\leq 2$ for $i=1,\dots,m$.
\end{proof}

\begin{proof}[Proof of Theorem \ref{thm:individual-Holder-exponents}]
Note that
\begin{equation*}
\tilde{\alpha}_f(t)=\sup\left\{\alpha>0: \frac{|f(t+h)-f(t)|}{|h|^\alpha}\to 0\ \mbox{as $h\to 0$}\right\}.
\end{equation*}
For an interval $I\subset[0,1]$, let $\omega_f(I):=\sup_{s,t\in I}|f(t)-f(s)|$ be the oscillation of $f$ over $I$.
Assume initially that $\lambda_{m}>0$, so that
\begin{equation}
K=\left(\frac{\log c_m}{\log c_1}\right)\log\lambda_1-\log\lambda_m.
\label{eq:formula-for-K}
\end{equation}

{\em (a)} We first show that $\tilde{\alpha}_f(t)\geq\alpha_0(t)$ when $\max\{\eps_1,\eps_m\}=1$. Let
\begin{equation*}
K_c:=\prod_{i=1}^m c_i^2, \qquad K_\lambda:=\prod_{i\in\II_+}\lambda_i^2.
\end{equation*}
Fix $\alpha<\alpha_0(t)$. Let $h>0$, and let $n$ be the largest integer such that
\begin{equation}
h\leq K_c|I_n(t)|=c_{i_1}\cdots c_{i_n}K_c.
\label{eq:h-estimate}
\end{equation}
Then
$h>c_{i_1}\cdots c_{i_{n+1}}K_c\geq c_{i_1}\cdots c_{i_{n}}c_{\min}K_c$, so there is an absolute constant $A>0$ such that $A^{-1}<h/c_{i_1}\cdots c_{i_n}<A$; we write this as $h \asymp c_{i_1}\cdots c_{i_n}$.

If $t+h\in I_n(t)$, then we have simply
\begin{equation}
\frac{|f(t+h)-f(t)|}{h^\alpha}\leq \frac{\omega_f(I_n(t))}{h^\alpha} \asymp
\frac{\lambda_{i_1}\cdots\lambda_{i_n}}{(c_{i_1}\cdots c_{i_n})^\alpha}=\prod_{i=1}^m \left(\frac{\lambda_i}{c_i^\alpha}\right)^{d_i(n;t)}\to 0
\label{eq:easy-case}
\end{equation}
as $n\to\infty$. 

Assume therefore that $t+h\not\in I_n(t)$. Let $j$ be the integer such that $I_n(t)=I_{n,j}$. Let $(i_1',\dots,i_n')$ be the coding of the adjacent interval $I_{n,j+1}$. We sort out the cases in which the length of $I_{n,j+1}$ is comparable to $h$. Define $k_i$ and $k_i'$ as in Lemma \ref{lem:boundedness}. Since $\max\{\eps_1,\eps_m\}=1$, Lemma \ref{lem:boundedness} implies that $|k_i-k_i'|\leq 2$ for each $i$. 
But then $|I_{n,j+1}|\geq K_c|I_{n,j}|\geq h$, so $t+h\in I_{n,j+1}$, and moreover, $\omega_f(I_{n,j+1})\leq K_\lambda^{-1}\omega_f(I_{n,j})$. Thus, 
\begin{equation}
\frac{|f(t+h)-f(t)|}{h^\alpha}\leq \frac{\omega_f(I_{n,j})+\omega_f(I_{n,j+1})}{h^\alpha}\leq \frac{(1+K_\lambda^{-1})\omega_f(I_n(t))}{h^\alpha}\to 0,
\label{eq:this-quotient-going-to-zero}
\end{equation}
as in \eqref{eq:easy-case}. The argument for $h<0$ is similar. Hence, $\tilde{\alpha}_f(t)\geq\alpha_0(t)$ when $\max\{\eps_1,\eps_m\}=1$.

\bigskip

{\em (b)} We next show that $\tilde{\alpha}_f(t)\geq\alpha_1(t)$ when $\eps_1=\eps_m=0$. Fix $\alpha<\alpha_1(t)$; then certainly $\alpha<\alpha_0(t)$, so we can follow the argument under {\em (a)} up to the point where Lemma \ref{lem:boundedness} is used. Define $n_0$, $k_i$ and $k_i'$ as in Lemma \ref{lem:boundedness}. If $n_0\geq n-1$, then obviously $|k_i-k_i'|\leq 2$ for each $i$. By Lemma \ref{lem:boundedness}, the same is true if $\eps_{i_{n_0}}\neq \eps_{i_{n_0}'}$. In these cases, we get \eqref{eq:this-quotient-going-to-zero} in the same way as before and we are done.

From now on we can, therefore, assume that $n_0\leq n-2$, $\eps_1=\eps_m=0$ and $\eps_{i_{n_0}}=\eps_{i_{n_0}'}$. Since $n_0<n$, $i_n\in\{1,m\}$.

\bigskip
{\em Case 1.} Suppose $i_n=m$. Then $i_\nu=m$ and $i_{\nu}'=1$ for $\nu=n_0+1,\dots,n$, so $l:=L_n(t)=n-n_0$.
Now $|I_n(t)|=c_{i_1}\cdots c_{i_n}=c_{i_1}\cdots c_{i_{n-l}}c_m^l$, and
\begin{equation*}
|I_{n,j+1}|=c_{i_1}\cdots c_{i_{n-l-1}}c_{i_{n-l}'}c_1^l.
\end{equation*}
Let $k$ be the largest integer (possibly negative) such that
\begin{equation*}
c_{i_1}\cdots c_{i_{n-l-1}}c_{i_{n-l}'}c_1^{l+k}\geq h.
\end{equation*}
Then
\begin{equation*}
c_1^{k} \asymp \frac{h}{c_{i_1}\cdots c_{i_{n-l-1}}c_{i_{n-l}'}c_1^l}\asymp \left(\frac{c_m}{c_1}\right)^l,
\end{equation*}
and so
\begin{equation}
\lambda_1^{k}\asymp \lambda_1^{l(\log c_m-\log c_1)/\log c_1}.
\label{eq:lambda-to-the-k-magnitude}
\end{equation}
Note that $k\geq -l$ by \eqref{eq:h-estimate}. Therefore, the interval $I_{n+k}(t+h)$ is adjacent to $I_n(t)$ and has length $c_{i_1}\cdots c_{i_{n-l-1}}c_{i_{n-l}'}c_1^{l+k}$. Moreover,
\begin{equation*}
\omega_f(I_{n+k}(t+h))=\omega_f([0,1])\lambda_{i_1}\cdots\lambda_{i_{n-l-1}}\lambda_{i_{n-l}'}\lambda_1^{l+k}.
\end{equation*}
At this point, the fact that $i_{n-l+1}=m$ and $i_{n-l+1}'=1$ implies that $(i_{n-l})_{\to}=i_{n-l}'$, so with $i:=i_{n-l}$, we have $\eps_i=\eps_{i_{\to}}$. Therefore, if $\chi_n(t)=0$, then $\lambda_{i_{n-l}'}=\lambda_{i_{\to}}=0$ and so $\omega_f(I_{n+k}(t+h))=0$, which means we simply have \eqref{eq:easy-case} again.

Suppose $\chi_n(t)=1$. Then $\lambda_{(i_{n-l})_{\to}}>0$, so
\begin{align*}
\omega_f(I_{n+k}(t+h))&\asymp \lambda_{i_1}\cdots \lambda_{i_{n-l}}\lambda_1^{l+k}
=\lambda_{i_1}\cdots\lambda_{i_n}\left(\frac{\lambda_1}{\lambda_m}\right)^l \lambda_1^k\\
&\asymp \lambda_{i_1}\cdots\lambda_{i_n}\left(\frac{\lambda_1^{\log c_m/\log c_1}}{\lambda_m}\right)^l
=\lambda_{i_1}\cdots\lambda_{i_n}e^{Kl},
\end{align*}
using \eqref{eq:lambda-to-the-k-magnitude} and \eqref{eq:formula-for-K}. It follows that
\begin{align*}
\frac{\omega_f(I_{n+k}(t+h))}{h^\alpha} &\asymp \frac{\lambda_{i_1}\cdots\lambda_{i_n}}{(c_{i_1}\cdots c_{i_n})^\alpha}e^{Kl}
=\left[e^{Kl/n}\prod_{i=1}^m \left(\frac{\lambda_i}{c_i^\alpha}\right)^{d_i(n;t)}\right]^n\\
&=\left[\exp\left\{\sum_{i=1}^m d_i(n;t)(\log\lambda_i-\alpha\log c_i)+K\chi_n(t)\frac{L_n(t)}{n}\right\}\right]^n\\
&\to 0,
\end{align*}
since $\alpha<\alpha_1(t)$. Hence,
\begin{equation}
\frac{|f(t+h)-f(t)|}{h^\alpha}\leq \frac{\omega_f(I_n(t))+\omega_f(I_{n+k}(t+h))}{h^\alpha}\to 0.
\label{eq:final-convergence}
\end{equation}

{\em Case 2.} Suppose $i_n=1$. This case is more straightforward. We reverse the roles of the ``digits" $1$ and $m$. If $t+h\not\in I_n(t)$, let $l'$ be the largest integer $j$ such that $i_{n-j+1}=\dots=i_{n-1}=i_n=1$ (or $l'=0$ if $i_n>1$). Put $K':=(\log c_1/\log c_m)\log\lambda_m-\log\lambda_1$, and note that $K'\leq 0$. Now we similarly find a basic interval $I_{n+k}(t+h)$ adjacent to $I_n(t)$ and with
\begin{equation*}
\omega_f(I_{n+k}(t+h))\asymp \lambda_{i_1}\cdots\lambda_{i_n}e^{K'l'}.
\end{equation*}
But, since $K'\leq 0$, we simply have $\omega_f(I_{n+k}(t+h))\leq C\lambda_{i_1}\cdots\lambda_{i_n}$ for a suitable constant $C$, and we again obtain \eqref{eq:final-convergence}.

The argument for $h<0$ is entirely similar. We have thus shown that $\tilde{\alpha}_f(t)\geq \alpha_1(t)$. 

\bigskip
{\em (c)} Next, we show that $\tilde{\alpha}_f(t)\leq \alpha_0(t)$. Let $\alpha>\alpha_0(t)$; then
\begin{equation*}
\limsup_{n\to\infty} \frac{|f(v_n)-f(u_n)|}{(v_n-u_n)^\alpha}=\limsup_{n\to\infty} \frac{\lambda_{i_1}\cdots\lambda_{i_n}}{(c_{i_1}\cdots c_{i_n})^\alpha}=\infty,
\end{equation*}
where we recall that $I_n(t)=[u_n,v_n]$. Since
\begin{equation*}
\frac{f(v_n)-f(u_n)}{(v_n-u_n)^\alpha}=\frac{f(v_n)-f(t)}{(v_n-t)^\alpha}\left(\frac{v_n-t}{v_n-u_n}\right)^\alpha-\frac{f(u_n)-f(t)}{(t-u_n)^\alpha}\left(\frac{t-u_n}{v_n-u_n}\right)^\alpha,
\end{equation*}
it follows that
\begin{equation*}
\limsup_{n\to\infty}\max\left\{\frac{|f(u_n)-f(t)|}{|u_n-t|^\alpha},\frac{|f(v_n)-f(t)|}{|v_n-t|^\alpha}\right\}=\infty.
\end{equation*}
Hence, $\tilde{\alpha}_f(t)\leq \alpha_0(t)$.

\bigskip
{\em (d)} Finally, we show that $\tilde{\alpha}_f(t)\leq \alpha_1(t)$ when $\eps_1=\eps_m=0$.
Fix $\alpha>\alpha_1(t)$. Given $n\in\NN$, let $l:=L_n(t)$, and let $j$ again be the integer such that $I_n(t)=I_{n,j}$. If $\chi_n(t)=0$, there is nothing to show, so assume $\chi_n(t)=1$. This implies $\eps_{i_{n-l}}=\eps_{(i_{n-l})_{\to}}$. Also assume $l\geq 1$, so $i_n=m$. (Otherwise there is again nothing to show.) Then either of the intervals $I_{n,j-1}$ and $I_{n,j+1}$ has length $c_{i_1}\cdots c_{i_{n-l-1}}c_{(i_{n-l})_{\to}}c_1^l$; assume without loss of generality that this interval is $I_{n,j+1}$.

Note that
\begin{equation}
c_{i_1}\cdots c_{i_{n-l-1}}c_{(i_{n-l})_{\to}}\geq c_{i_1}\cdots c_{i_{n-l}}K_c\geq c_{i_1}\cdots c_{i_n}K_c=K_c|I_n(t)|.
\label{eq:comparing-lengths}
\end{equation}
Let $k$ be the largest integer such that
\begin{equation*}
c_{i_1}\cdots c_{i_{n-l-1}}c_{(i_{n-l})_{\to}}c_1^{l+k}\geq K_c|I_n(t)|.
\end{equation*}
Then $k\geq -l$ by \eqref{eq:comparing-lengths}. Let $j_1$ be the integer such that $I_{n,j+1}$ and $I_{n+k,j_1}$ have the same left endpoint.
We have
\begin{equation}
|I_{n+k,j_1}|=c_{i_1}\cdots c_{i_{n-l-1}}c_{(i_{n-l})_{\to}}c_1^{l+k}\geq K_c|I_n(t)|.
\label{eq:at-least-this-long}
\end{equation}
Furthermore,
\begin{equation*}
\omega_f(I_{n+k,j_1})=\omega_f([0,1])\lambda_{i_1}\cdots\lambda_{i_{n-l-1}}\lambda_{(i_{n-l})_{\to}}\lambda_1^{l+k}.
\end{equation*}
Let $w_n$ be the right endpoint of $I_{n+k,j_1}$.
Since $\chi_n(t)=1$, we have $\lambda_{(i_{n-l})_{\to}}>0$, and as in the first half of the proof we obtain
\begin{equation*}
|f(w_n)-f(v_n)|\asymp\omega_f(I_{n+k,j_1})\asymp \lambda_{i_1}\cdots\lambda_{i_n}e^{Kl}=\lambda_{i_1}\cdots\lambda_{i_n}e^{K\chi_n(t)L_n(t)}.
\end{equation*}
Recall also that $w_n-v_n=|I_{n+k,j_1}|\asymp |I_n(t)|=c_{i_1}\cdots c_{i_n}$. Hence, since $\alpha>\alpha_1(t)$,
\begin{equation}
\limsup_{n\to\infty}\frac{|f(w_n)-f(v_n)|}{(w_n-v_n)^\alpha}=\infty.
\label{eq:this-one-or-that-one}
\end{equation}
Now write
\begin{equation*}
\frac{f(w_n)-f(v_n)}{(w_n-v_n)^\alpha}=\frac{f(w_n)-f(t)}{(w_n-t)^\alpha}\left(\frac{w_n-t}{w_n-v_n}\right)^\alpha
  -\frac{f(v_n)-f(t)}{(v_n-t)^\alpha}\left(\frac{v_n-t}{w_n-v_n}\right)^\alpha.
\end{equation*}
Note that
\begin{equation*}
0\leq\frac{v_n-t}{w_n-v_n}\leq\frac{w_n-t}{w_n-v_n}\leq\frac{|I_n(t)|+|I_{n+k,j_1}|}{|I_{n+k,j_1}|}\leq 1+K_c^{-1},
\end{equation*}
in view of \eqref{eq:at-least-this-long}. Therefore, \eqref{eq:this-one-or-that-one} implies
\begin{equation*}
\limsup_{n\to\infty}\max\left\{\frac{|f(v_n)-f(t)|}{(v_n-t)^\alpha},\frac{|f(w_n)-f(t)|}{(w_n-t)^\alpha}\right\}=\infty.
\end{equation*}
Hence, $\tilde{\alpha}_f(t)\leq \alpha_1(t)$.

\bigskip
{\em (e)} So far we had assumed that $\lambda_m>0$. However, if $\lambda_m=0$, then $K=0$ and our assumption that $i_n\in\mathcal{I}_+$ for each $n$ implies that $L_n(t)=0$ for every $n$, so we always have $l=0$ in the analysis above, and $\tilde{\alpha}_f(t)=\alpha_0(t)$. This completes the proof.
\end{proof}

Since $\mathcal{T}_0$ is a countable set, it plays no role in the determination of the pointwise H\"older spectrum of $f$. But, for completeness, we determine $\tilde{\alpha}_f(t)$ for $t\in\mathcal{T}_0$ as well. Each such $t$ has two different codings, either (i) both ending in $1^\infty$; (ii) both ending in $m^\infty$; (iii) one ending in $1^\infty$ and one in $m^\infty$; or (iv) both ending in $(1,m)^\infty$, either in sync or one step out of phase. (Possibility (iv) occurs when $\eps_1=\eps_m=1$.)

\begin{theorem} \label{thm:pointwise-Holder-at-endpoints}
Assume \eqref{eq:ratio-assumption}, and let $t\in\mathcal{T}_0$. Assume neither coding of $t$ contains a digit from $\II_0$.
\begin{enumerate}[(i)]
\item If at least one coding of $t$ ends in $1^\infty$, then $\tilde{\alpha}_f(t)=\log\lambda_1/\log c_1$;
\item If both codings of $t$ end in $m^\infty$, then $\tilde{\alpha}_f(t)=\log\lambda_m/\log c_m$;
\item If both codings of $t$ end in $(1,m)^\infty$, then $\tilde{\alpha}_f(t)=\log(\lambda_1\lambda_m)/\log(c_1 c_m)$.
\end{enumerate}
\end{theorem}

\begin{proof}
Assume $t$ has one coding ending in $1^\infty$ and one ending in $m^\infty$. Say the coding ``from the right" is the one ending in $1^\infty$; that is, for some interval $I:=I_{i_1,\dots,i_N}$ and all $k\in\NN$, $t$ is the left endpoint of $I_{i_1,\dots,i_N,1^k}$. Let $0<h\leq |I|$, and let $n$ be the integer such that $|I|c_1^{n+1}<h\leq |I|c_1^n$. Then
\begin{equation*}
|f(t+h)-f(t)|\leq \omega_f(I)\omega_f([0,c_1^n])=\omega_f(I)\omega_f([0,1])\lambda_1^n\leq K_1 h^{\log\lambda_1/\log c_1},
\end{equation*}
for a suitable constant $K_1$ depending only on $t$. For $h<0$ we similarly have $|f(t+h)-f(t)|\leq K_2|h|^{\log\lambda_{m}/\log c_m}$. As a result, $\tilde{\alpha}_f(t)\geq \min\{\log\lambda_1/\log c_1,\log\lambda_{m}/\log c_m\}=\log\lambda_1/\log c_1$. Equality follows by considering the sequence $h=|I|c_1^{n}$. 

The other cases are proved similarly.
\end{proof}

\begin{remark}
{\rm
{\em (a)} It follows from Theorem \ref{thm:pointwise-Holder-at-endpoints} that the expression \eqref{eq:Holder-exponent-if-densities-exist} is valid also for $t\in\mathcal{T}_0$, provided that in case (i) we use a coding for $t$ ending in $1^\infty$.

{\em (b)} If both codings of $t$ contain a digit from $\II_0$, then $\tilde{\alpha}_f(t)=\infty$. If only one coding of $t$ contains a digit from $\II_0$, then the other coding determines $\tilde{\alpha}_f(t)$ as in cases (i)-(iii) of Theorem \ref{thm:pointwise-Holder-at-endpoints} (suitably modified).
}
\end{remark}

\section{Proofs of Theorems \ref{thm:Hausdorff-dimension} and \ref{thm:Holder-spectrum}} \label{sec:Holder-proofs}

The lower bound for the dimension of $\tilde{E}_f(\alpha)$ is straightforward.

\begin{lemma} \label{lem:lower-bound}
Let $\alpha_{\min}\leq\alpha\leq\alpha_{\max}$. Then $\dim_H \tilde{E}_f(\alpha)\geq \beta^*(\alpha)$.
\end{lemma}

\begin{proof}
Let $\mathbf{p}^*=(p_1^*,\dots,p_m^*)$ attain the constrained maximum in \eqref{eq:duality}. By Corollary \ref{cor:holder-exponent-densities}, $\tilde{E}_f(\alpha)$ contains the set of points $t\in(0,1)$ such that $d_i(t)=p_i^*$ for each $i$. Hence, by \eqref{eq:general-Eggleston} and Proposition \ref{prop:duality}, $\dim_H \tilde{E}_f(\alpha)\geq H(\mathbf{p}^*)=\beta^*(\alpha)$. (Note that the only situation in which Corollary \ref{cor:holder-exponent-densities} does not apply is when $p_{m}^*=1$, but in that case $\beta^*(\alpha)=0$ and the lower bound holds trivially.)
\end{proof}

Proving that $\dim_H \tilde{E}_f(\alpha)\leq \beta^*(\alpha)$ is rather more difficult. We first develop a few technical lemmas. Without loss of generality we assume \eqref{eq:ratio-assumption}. For brevity, write
\begin{equation*}
\gamma_i:=\log\lambda_i-\alpha\log c_i, \qquad i=1,\dots,m,
\end{equation*}
where we set $\log 0\equiv -\infty$. In the following definition and lemmas, assume $K$ is given by \eqref{eq:formula-for-K}. For $n\in\NN$, $l=0,1,\dots,n$ and $\eps>0$, define the set of partitions
\begin{align*}
\mathcal{P}_{n,l}^{(\eps)}&:=\bigg\{(k_1,\dots,k_{m})\in\{0,1,\dots,n\}^m: \sum_{i=1}^{m}k_i=n,\ k_{m}\geq l,\ \mbox{and}\\
&\qquad\qquad\qquad\qquad \sum_{i=1}^{m}k_i\gamma_i+Kl\geq-n\eps\bigg\}.
\end{align*}

\begin{lemma} \label{lem:technical1}
Suppose $c_1\geq c_m$. Then for every $s>0$, every $\eps>0$, every sufficiently large $n$ and every $l\in\{0,1,\dots,n\}$,
\begin{align*}
\max&\left\{\frac{(n-l)!}{k_1!\cdots k_{m-1}!(k_m-l)!}\prod_{i=1}^m c_i^{k_i s}: (k_1,\dots,k_m)\in\mathcal{P}_{n,l}^{(\eps)}\right\}\\
&\qquad \leq \max\left\{\frac{(n+n_0)!}{k_1!\cdots k_m!}\prod_{i=1}^m c_i^{k_i s}: (k_1,\dots,k_m)\in\mathcal{P}_{n+n_0,0}^{(2\eps)}\right\},
\end{align*}
where $n_0\geq 0$ is the largest integer such that $c_m^l\leq c_1^{l+n_0}$.
\end{lemma}

\begin{proof}
Let $(k_1,\dots,k_m)\in\mathcal{P}_{n,l}^{(\eps)}$. Put $k_1':=k_1+l+n_0$, $k_i':=k_i$ for $i=2,\dots,m-1$, and $k_m':=k_m-l$. Then $\sum_{i=1}^m k_i'=n+n_0$, and $k_m'\geq 0$. Moreover,
\begin{align*}
\sum_{i=1}^m k_i'\gamma_i&=\sum_{i=1}^m k_i\gamma_i+(l+n_0)\gamma_1-l\gamma_m\\
&=\sum_{i=1}^m k_i\gamma_i+Kl+(l+n_0)(\log\lambda_1-\alpha\log c_1)-l(\log\lambda_m-\alpha\log c_m)\\
&\qquad\qquad\qquad -l\left\{\left(\frac{\log c_m}{\log c_1}\right)\log\lambda_1-\log\lambda_m\right\}\\
&=\sum_{i=1}^m k_i\gamma_i+Kl+\{(l+n_0)\log c_1-l\log c_m\}\left(\frac{\log\lambda_1}{\log c_1}-\alpha\right)\\
&\geq -n\eps-|\gamma_1|\geq -2n\eps\geq -2(n+n_0)\eps
\end{align*}
for all sufficiently large $n$, where the first inequality follows since $(k_1,\dots,k_m)\in\mathcal{P}_{n,l}^{(\eps)}$ and $0\leq (l+n_0)\log c_1-l\log c_m\leq -\log c_1$, by definition of $n_0$. Thus, $(k_1',\dots,k_m')\in \mathcal{P}_{n+n_0,0}^{(2\eps)}$ for all large enough $n$. Furthermore, by the definition of $\mathcal{P}_{n,l}^{(\eps)}$, $n\geq k_1+k_m\geq k_1+l$, and so $k_1\leq n-l$. Hence,
\begin{equation*}
\frac{(n-l)!}{k_1!\cdots k_{m-1}!(k_m-l)!}\leq \frac{(n+n_0)!}{(k_1+l+n_0)!k_2!\cdots k_{m-1}!(k_m-l)!}=\frac{(n+n_0)!}{k_1'!\cdots k_m'!},
\end{equation*}
and
\begin{equation*}
\prod_{i=1}^m c_i^{k_i s}=\left(\prod_{i=1}^m c_i^{k_i's}\right)c_m^{ls} c_1^{-(l+n_0)s}\leq \prod_{i=1}^m c_i^{k_i's},
\end{equation*}
again by definition of $n_0$. Thus, the desired inequality follows.
\end{proof}

\begin{lemma} \label{lem:technical2}
Suppose $c_1<c_m$. Then for every $s>0$, every $\eps>0$, every sufficiently large $n$ and every $l\in\{0,1,\dots,n\}$,
\begin{align*}
\max&\left\{\frac{(n-l)!}{k_1!\cdots k_{m-1}!(k_m-l)!}\prod_{i=1}^m c_i^{k_i s}: (k_1,\dots,k_m)\in\mathcal{P}_{n,l}^{(\eps)}\right\}\\
&\qquad \leq \max\left\{\frac{(n-n_0)!}{k_1!\cdots k_m!}\prod_{i=1}^m c_i^{k_i s}: (k_1,\dots,k_m)\in\mathcal{P}_{n-n_0,0}^{(\delta)}\right\},
\end{align*}
where $n_0\leq l$ is the smallest integer such that $c_m^l\leq c_1^{l-n_0}$, and $\delta:=4\eps\log c_1/\log c_m$.
\end{lemma}

\begin{proof}
Let $(k_1,\dots,k_m)\in\mathcal{P}_{n,l}^{(\eps)}$. Put $k_1':=k_1+l-n_0$, $k_i':=k_i$ for $i=2,\dots,m-1$, and $k_m':=k_m-l$. Then $\sum_{i=1}^m k_i'=n-n_0$, and $k_m'\geq 0$. Moreover, as in the proof of Lemma \ref{lem:technical1},
\begin{align*}
\sum_{i=1}^m k_i'\gamma_i&=\sum_{i=1}^m k_i\gamma_i+(l-n_0)\gamma_1-l\gamma_m\\
&=\sum_{i=1}^m k_i\gamma_i+Kl+\{(l-n_0)\log c_1-l\log c_m\}\left(\frac{\log\lambda_1}{\log c_1}-\alpha\right)\\
&\geq -n\eps-|\gamma_1|\geq -2n\eps
\end{align*}
for all sufficiently large $n$. Now the definition of $n_0$ implies $(l-n_0+1)\log c_1<l\log c_m$, so
\begin{equation*}
n_0-1<l\left(1-\frac{\log c_m}{\log c_1}\right)\leq n\left(1-\frac{\log c_m}{\log c_1}\right),
\end{equation*}
and hence,
\begin{equation*}
n-n_0\geq \left(\frac{\log c_m}{\log c_1}\right)n-1\geq \left(\frac{\log c_m}{2\log c_1}\right)n
\end{equation*}
for all large enough $n$, keeping in mind that $0<\log c_m/\log c_1<1$. Thus, $2n\eps\leq (n-n_0)\delta$, and it follows that $(k_1',\dots,k_m')\in\mathcal{P}_{n-n_0,0}^{(\delta)}$ for all sufficiently large $n$. Furthermore (since $n_0\leq l$),
\begin{equation*}
\frac{(n-l)!}{k_1!\cdots k_{m-1}!(k_m-l)!}\leq \frac{(n-n_0)!}{(k_1+l-n_0)!k_2!\cdots k_{m-1}!(k_m-l)!}=\frac{(n-n_0)!}{k_1'!\cdots k_m'!},
\end{equation*}
and
\begin{equation*}
\prod_{i=1}^m c_i^{k_i s}=\left(\prod_{i=1}^m c_i^{k_i's}\right)c_m^{ls} c_1^{-(l-n_0)s}\leq \prod_{i=1}^m c_i^{k_i's},
\end{equation*}
again by definition of $n_0$. The desired inequality follows.
\end{proof}

\begin{lemma} \label{lem:multinomial-estimate}
Let $n\in\NN$, and let $k_1,\dots,k_m$ be nonnegative integers with $\sum_{i=1}^m k_i=n$. Put $p_i:=k_i/n$ for $i=1,\dots,m$. Then
\begin{equation*}
\frac{n!}{k_1!\cdots k_m!}\leq 2\sqrt{n}\left(\prod_{i=1}^m p_i^{-p_i}\right)^n,
\end{equation*}
where we use the convention $0^0\equiv 1$.
\end{lemma}

\begin{proof}
We use the following precise version of Stirling's approximation:
\begin{equation}
\sqrt{2\pi}n^{n+\frac12}e^{-n}\leq n!\leq en^{n+\frac12}e^{-n}, \qquad\mbox{for all $n\in\NN$}.
\label{eq:Stirling}
\end{equation}
Let $\mathcal{E}_+:=\{i:k_i>0\}$. Using that $e/\sqrt{2\pi}\leq 2$, we obtain from \eqref{eq:Stirling}
\begin{align*}
\frac{n!}{k_1!\cdots k_{m}!}&=\frac{n!}{\prod_{i\in\mathcal{E}_+}k_i!} 
	\leq \frac{2n^{n+\frac12}e^{-n}}{\prod_{i\in\mathcal{E}_+}k_i^{k_i+\frac12}e^{-k_i}}
  \leq 2\sqrt{n}\frac{n^n}{\prod_{i\in\mathcal{E}_+}k_i^{k_i}}\\
&=2\sqrt{n}\left(\prod_{i\in\mathcal{E}_+}p_i^{-p_i}\right)^{n}=2\sqrt{n}\left(\prod_{i=1}^{m}p_i^{-p_i}\right)^{n},
\end{align*}
as required.
\end{proof}

Next, let $\mathbf{c}^*=(c_1^*,\dots,c_m^*)$ be the vector defined by $c_i^*=0$ for $i\in\II_0$, and $c_i^*=c_i^{\hat{s}}$ for $i\in\II_+$. Then $\mathbf{c}^*\in\Delta_m^0$, and
\begin{equation*}
\max\{H(\mathbf{p}): \mathbf{p}\in\Delta_m^0\}=H(\mathbf{c}^*)=\hat{s}.
\end{equation*}
Note also that, if $\alpha<\hat{\alpha}$, then
\begin{equation*}
\sum_{i\in\II_+}c_i^*\gamma_i=\sum_{i\in\II_+}c_i^*(\log\lambda_i-\alpha\log c_i)
<\sum_{i\in\II_+}c_i^{\hat{s}}(\log\lambda_i-\hat{\alpha}\log c_i)=0,
\end{equation*}
so the half space $\{\mathbf{p}:\sum p_i\gamma_i\geq 0\}$ does not contain the point $\mathbf{c}^*$. For $\delta>0$, define the subsimplex
\begin{equation*}
\Gamma_{\delta}:=\Delta_m^0\cap\left\{\mathbf{p}:\sum p_i\gamma_i\geq -\delta\right\}.
\end{equation*}

\begin{lemma} \label{lem:subsimplex-maximum}
Let $\delta>0$ and suppose $\mathbf{c}^*\not\in\{\mathbf{p}:\sum p_i\gamma_i\geq -\delta\}$.
Then the maximum value of $H$ over $\Gamma_{\delta}$ is attained on the hyperplane $\{\mathbf{p}:\sum p_i\gamma_i=-\delta\}$.
\end{lemma}

\begin{proof}
Using Lagrange multipliers one can show that $H$ has only one critical point in $\Delta_m^0$, which is the absolute maximum at $\mathbf{c}^*$. Since this point lies outside $\Gamma_{\delta}$, the maximum of $H$ over $\Gamma_{\delta}$ is attained on the boundary of $\Gamma_{\delta}$, where we use the relative topology of the $(\#\II_+-1)$-dimensional hyperplane containing $\Delta_m^0$. But it is easy to see that the restriction of $H$ to $\Delta_m^0$ cannot have a local maximum on the boundary of $\Delta_m^0$, because $H$ has infinite gradient there. Hence, the maximum value of $H$ on $\Gamma_{\delta}$ is attained on the set $\{\mathbf{p}:\sum p_i\gamma_i=-\delta\}$.
\end{proof}

\begin{proposition} \label{prop:Holder-halves}
Let $\alpha_{\min}\leq\alpha\leq\alpha_{\max}$.
\begin{enumerate}[(i)]
\item If $\alpha\leq \hat{\alpha}$, then $\dim_H\{t:\tilde{\alpha}_f(t)\leq\alpha\}=\beta^*(\alpha)$;
\item If $\alpha\geq \hat{\alpha}$, then $\dim_H\{t:\tilde{\alpha}_f(t)\geq\alpha\}=\beta^*(\alpha)$.
\end{enumerate}
\end{proposition}

\begin{proof}
By Lemma \ref{lem:lower-bound}, $\beta^*(\alpha)$ is a lower bound for the Hausdorff dimension of each of the above two sets, since they both contain $\tilde{E}_f(\alpha)$. The proof that $\beta^*(\alpha)$ is also an upper bound is more involved.

Suppose $\alpha\leq \hat{\alpha}$, and consider the set $F(\alpha):=\{t:\tilde{\alpha}_f(t)\leq\alpha\}$. Note that $F(\alpha)\subset \KK_\phi$. If $\alpha=\hat{\alpha}$, then $\beta^*(\alpha)=\hat{s}=\dim_H\KK_\phi\geq \dim_H F(\alpha)$. Assume therefore that $\alpha<\hat{\alpha}$. Without loss of generality, we may also assume \eqref{eq:ratio-assumption}. 

Assume first that $\lambda_m>0$. For $\delta>0$, define
\begin{equation}
\beta_\delta^*(\alpha):=\max\bigg\{H(\mathbf{p}): \mathbf{p}=(p_1,\dots,p_{m})\in\Delta_m^0,\ \sum_{i\in\mathcal{I_+}}p_i\gamma_i=-\delta\bigg\}.
\label{eq:beta-delta}
\end{equation}
Note that $\beta_\delta^*(\alpha)\to\beta^*(\alpha)$ as $\delta\downarrow 0$.

Given $\eps>0$, let $\delta(\eps):=\max\{2\eps,4\eps\log c_1/\log c_m\}$. Fix $s>\beta^*(\alpha)$, and choose $\eps>0$ small enough so that in fact, $s>\beta_{\delta(\eps)}^*(\alpha)$ and $\mathbf{c}^*\not\in\{\mathbf{p}:\sum p_i\gamma_i\geq -\delta(\eps)\}$. (This is possible by the discussion preceding Lemma \ref{lem:subsimplex-maximum}.) Set $\delta:=\delta(\eps)$. 

For $(k_1,\dots,k_m)\in\mathcal{P}_{n,0}^{(\delta)}$, put $p_i:=k_i/n$ for $i=1,\dots,m$, and note that $\mathbf{p}=(p_1,\dots,p_m)\in\Gamma_{\delta}$, since the condition $\sum k_i\gamma_i\geq -n\delta$ forces those $k_i$'s with $\lambda_i=0$ to be zero. By Lemma \ref{lem:multinomial-estimate},
\begin{equation}
\frac{n!}{k_1!\cdots k_m!}\prod_{i=1}^m c_i^{k_i s}\leq 2\sqrt{n}\left(\prod_{i=1}^m p_i^{-p_i}c_i^{p_i s}\right)^n,
\label{eq:combination-estimate}
\end{equation}
and the key is that the product inside the $n$th power is less than 1:
\begin{equation*}
\mathbf{p}\in\Gamma_\delta \quad \Rightarrow \quad \prod_{i=1}^m p_i^{-p_i}c_i^{p_i s}<\prod_{i=1}^m p_i^{-p_i}c_i^{p_i \beta_\delta^*(\alpha)}\leq \prod_{i=1}^m p_i^{-p_i}c_i^{p_i H(\mathbf{p})}=1.
\end{equation*}
Since $\Gamma_\delta$ is compact, it follows that
\begin{equation}
\max\left\{\prod_{i=1}^m p_i^{-p_i}c_i^{p_i s}: \mathbf{p}\in\Gamma_\delta\right\}<1.
\label{eq:product-less-than-one}
\end{equation}

Now fix $N\in\NN$. If $t\in F(\alpha)\backslash \mathcal{T}_0$, then there is in integer $n\geq N$ such that 
\begin{equation*}
\sum_{i=1}^m k_i(n;t)\gamma_i+KL_n(t)\geq -n\eps,
\end{equation*}
by Theorem \ref{thm:individual-Holder-exponents}, so the $m$-tuple $(k_1,\dots,k_m):=(k_1(n;t),\dots,k_m(n;t))$ lies in $\mathcal{P}_{n,l}^{(\eps)}$, where $l=L_n(t)$. There are
\begin{equation*}
\frac{(n-l)!}{k_1!\cdots k_{m-1}!(k_m-l)!}
\end{equation*}
basic intervals of level $n$ whose points satisfy $L_n(t)\geq l$ and $k_i(n;t)=k_i$ for $i=1,\dots,m$, and each has length $\prod_{i=1}^m c_i^{k_i}$. Let $\HH^s$ denote $s$-dimensional Hausdorff measure. Then (since $\mathcal{T}_0$ is countable),
\begin{equation*}
\HH^s(F(\alpha))\leq \sum_{n=N}^\infty \sum_{l=0}^n \sum\left\{\frac{(n-l)!}{k_1!\cdots k_{m-1}!(k_m-l)!}\prod_{i=1}^m c_i^{k_i s}: (k_1,\dots,k_m)\in \mathcal{P}_{n,l}^{(\eps)}\right\}.
\end{equation*}
We now consider two cases.

\bigskip
{\em Case 1.} Suppose $c_1\geq c_m$. For each $n\geq N$ and $l\leq n$, let $n_0=n_0(l)$ be as in Lemma \ref{lem:technical1}. Then
\begin{equation*}
n_0(l)\leq l\left(\frac{\log c_m}{\log c_1}-1\right)\leq n\left(\frac{\log c_m}{\log c_1}-1\right)=:K_1 n,
\end{equation*}
where $K_1\geq 0$. Since $\Gamma_\eps\subset \Gamma_{\eps'}$ for $\eps<\eps'$ and $\#\mathcal{P}_{n,l}^{(\eps)}\leq (n+1)^m$, Lemma \ref{lem:technical1}, \eqref{eq:combination-estimate} and \eqref{eq:product-less-than-one} give
\begin{align*}
\HH^s(F(\alpha)) &\leq \sum_{n=N}^\infty (n+1)^m\sum_{l=0}^n \max\left\{\frac{(n+n_0)!}{k_1!\cdots k_m!}\prod_{i=1}^m c_i^{k_i s}: (k_1,\dots,k_m)\in\mathcal{P}_{n+n_0,0}^{(2\eps)}\right\}\\
&\leq \sum_{n=N}^\infty (n+1)^m\sum_{l=0}^n \max\left\{2\sqrt{n+n_0(l)}\left(\prod_{i=1}^m p_i^{-p_i}c_i^{p_i s}\right)^{n+n_0(l)}: \mathbf{p}\in\Gamma_{2\eps}\right\}\\
&\leq 2\sum_{n=N}^\infty (n+1)^{m+1}\sqrt{(1+K_1)n} \left(\max\left\{\prod_{i=1}^m p_i^{-p_i}c_i^{p_i s}: \mathbf{p}\in\Gamma_{\delta}\right\}\right)^n\\
&\to 0 \qquad\mbox{as $N\to\infty$}.
\end{align*}

\bigskip
{\em Case 2.} Suppose $c_1<c_m$. For each $n\geq N$ and $l\leq n$, let $n_0=n_0(l)$ be as in Lemma \ref{lem:technical2}. Then
\begin{equation*}
n-n_0(l)\geq \left(\frac{\log c_m}{2\log c_1}\right)n=:K_2 n
\end{equation*}
for all large enough $n$, so Lemma \ref{lem:technical2}, \eqref{eq:combination-estimate} and \eqref{eq:product-less-than-one} give
\begin{align*}
\HH^s(F(\alpha))&\leq \sum_{n=N}^\infty (n+1)^m\sum_{l=0}^n \max\left\{2\sqrt{n-n_0(l)}\left(\prod_{i=1}^m p_i^{-p_i}c_i^{p_i s}\right)^{n-n_0(l)}: \mathbf{p}\in\Gamma_{\delta}\right\}\\
&\leq 2\sum_{n=N}^\infty (n+1)^{m+1}\sqrt{n} \left(\max\left\{\prod_{i=1}^m p_i^{-p_i}c_i^{p_i s}: \mathbf{p}\in\Gamma_{\delta}\right\}\right)^{K_2 n}\\
&\to 0 \qquad\mbox{as $N\to\infty$}.
\end{align*}
This concludes the proof that $\dim_H F(\alpha)\leq\beta^*(\alpha)$ in case $\lambda_m>0$.

When $\lambda_m=0$ the proof is much simpler, since $K=0$ and so the term in \eqref{eq:general-Holder-limsup-equation} involving $L_n(t)$ vanishes, leading to a much more straightforward covering argument.

Likewise, the proof of (ii) is much easier, because $\tilde{\alpha}_f(t)\leq \alpha_0(t)$, where $\alpha_0(t)$ is defined as in Theorem \ref{thm:individual-Holder-exponents}, and so we have $\{t:\tilde{\alpha}_f(t)\geq\alpha\}\subset \{t:\alpha_0(t)\geq\alpha\}$. This last set is straightforward to cover in a manner similar to the above.

Alternatively, (ii) may be deduced quickly from the main result of \cite{LiWuXiong}.
\end{proof}

\begin{proof}[Proof of Theorem \ref{thm:Holder-spectrum}]
(i) Theorems \ref{thm:individual-Holder-exponents} and \ref{thm:pointwise-Holder-at-endpoints} imply that $\alpha_{\min}\leq\tilde{\alpha}_f(t)\leq\alpha_{\max}$ or $\tilde{\alpha}_f(t)=\infty$ for every $t$.

(ii) If $\mathcal{I}_0=\emptyset$, then $\tilde{\alpha}_f(t)<\infty$ for every $t$ by Theorems \ref{thm:individual-Holder-exponents} and \ref{thm:pointwise-Holder-at-endpoints}. If $\mathcal{I}_0\neq\emptyset$, then the standard coding of almost every $t\in(0,1)$ has at least one digit in $\mathcal{I}_0$, and so $\tilde{\alpha}_f(t)=\infty$ almost everywhere.

(iii) Fix $\alpha\in(\alpha_{\min},\alpha_{\max})$. By Lemma \ref{lem:lower-bound}, $\dim_H \tilde{E}_f(\alpha)\geq \beta^*(\alpha)$. The reverse inequality follows from Proposition \ref{prop:Holder-halves}: If $\alpha\leq\hat{\alpha}$, use that $\tilde{E}_f(\alpha)\subset \{t:\tilde{\alpha}_f(t)\leq\alpha\}$; and if $\alpha>\hat{\alpha}$, use that $\tilde{E}_f(\alpha)\subset \{t:\tilde{\alpha}_f(t)\geq\alpha\}$.

(iv) When $\alpha=\alpha_{\min}$, the constraint in \eqref{eq:duality} holds if and only if $p_i=0$ for each $i$ with $\rho_i=\log\lambda_i/\log c_i>\alpha_{\min}$. The constrained maximum is then attained when $p_i=c_i^{s_{\min}}$ for each $i$ with $\rho_i=\alpha_{\min}$. But then $\dim_H \tilde{E}_f(\alpha_{\min})=H(p_1,\dots,p_m)=s_{\min}$, using Proposition \ref{prop:duality}. The dimension of $\tilde{E}_f(\alpha_{\max})$ follows similarly.

(v) Since $\tilde{\alpha}_f(t)<\infty$ only for $t\in \KK_\phi$, it follows that $\dim_H \tilde{E}_f(\alpha)\leq \dim_H \KK_\phi=\hat{s}$ for all $\alpha\in[\alpha_{\min},\alpha_{\max}]$. The bound is attained for $\alpha=\hat{\alpha}$, as we can take $(p_1,\dots,p_{m})$ in \eqref{eq:duality} to be the $m$-tuple given by $p_i=0$ for $i\in\mathcal{I}_0$, and $p_i=c_i^{\hat{s}}$ for $i\in\mathcal{I}_+$. Moreover, if $\mathcal{I}_0=\emptyset$, then $\tilde{\alpha}_f(t)=\hat{\alpha}$ for all $t$ such that $d_i(t)=c_i$ for each $i$, in view of Corollary \ref{cor:holder-exponent-densities}; hence, $\tilde{\alpha}_f(t)=\hat{\alpha}$ for almost all $t$.
\end{proof}

\begin{proof}[Proof of Theorem \ref{thm:Hausdorff-dimension}]
Assume throughout that $\lambda_i<c_i$ for at least one $i$.

(i) Suppose $\sum_{i=1}^{m}c_i\log(\lambda_i/c_i)\geq 0$. Then in particular, $\lambda_i>0$ for each $i$, and
\begin{equation*}
\hat{\alpha}=\frac{\sum_{i=1}^m c_i\log\lambda_i}{\sum_{i=1}^m c_i\log c_i}\leq 1.
\end{equation*}
Moreover, $\lambda_j>c_j$ for at least one $j$, so $\alpha_{\min}<1<\alpha_{\max}$. Since
\begin{equation*}
\{t:\tilde{\alpha}_f(t)\geq 1+\eps\}\subset\mathcal{D}(f)\subset\{t:\tilde{\alpha}_f(t)\geq 1\} \quad\mbox{for every $\eps>0$},
\end{equation*}
Proposition \ref{prop:Holder-halves}(ii) and the continuity of $\beta^*(\alpha)$ imply $\dim_H \mathcal{D}(f)=\beta^*(1)>0$.

(ii) Suppose $\sum_{i\in\II_+}c_i^{\hat{s}}\log(\lambda_i/c_i)\geq 0$. Then $\hat{\alpha}\leq 1$, so $\tilde{E}_f(\hat{\alpha}) \subset \{t:\tilde{\alpha}_f(t)\leq 1\} \subset \KK_\phi$. Thus, $\dim_H \{t:\tilde{\alpha}_f(t)\leq 1\}=\hat{s}$. By a straightforward continuity argument, $\dim_H \{t:\tilde{\alpha}_f(t)\leq 1-\eps\}\to \dim_H \{t:\tilde{\alpha}_f(t)\leq 1\}$ as $\eps\downarrow 0$. Since
\begin{equation}
\{t:\tilde{\alpha}_f(t)\leq 1-\eps\}\subset\mathcal{D}_\sim(f)\subset\{t:\tilde{\alpha}_f(t)\leq 1\} \quad\mbox{for every $\eps>0$},
\label{eq:D-sandwich}
\end{equation}
we conclude that $\dim_H \mathcal{D}_\sim(f)=\hat{s}>0$.

(iii) Suppose $\sum_{i\in\II_+}c_i^{\hat{s}}\log(\lambda_i/c_i)<0$. Then as in (i) above, $\alpha_{\min}<1<\alpha_{\max}$, but now $\hat{\alpha}>1$. 
Using \eqref{eq:D-sandwich}, Proposition \ref{prop:Holder-halves}(i) implies $\dim_H \mathcal{D}_\sim(f)=\beta^*(1)>0$.
\end{proof}




\section{Multifractal formalism for self-similar measures} \label{sec:multifractal}

In this section, consider a self-similar measure $\mu$ on $[0,1]$ defined as in Example \ref{ex:self-similar-measure}. We are interested in the upper and lower local dimension of $\mu$ at a point $t\in(0,1)$, defined by
\begin{equation*}
\overline{\alpha}_\mu(t):=\limsup_{r\to 0}\frac{\log \mu(B(t,r))}{\log r}, \qquad 
\underline{\alpha}_\mu(t):=\liminf_{r\to 0}\frac{\log \mu(B(t,r))}{\log r},
\end{equation*}
where $B(t,r):=(t-r,t+r)$. If $\overline{\alpha}_\mu(t)=\underline{\alpha}_\mu(t)$, we denote the common value by $\alpha_\mu(t)$.
Let
\begin{gather*}
\overline{E}_\mu(\alpha):=\{t:\overline{\alpha}_\mu(t)=\alpha\}, \\
\underline{E}_\mu(\alpha):=\{t:\underline{\alpha}_\mu(t)=\alpha\}, \\
E_\mu(\alpha):=\{t:\alpha_\mu(t)=\alpha\}.
\end{gather*}
It has been known for some time (see \cite{Arbeiter}) that
\begin{equation*}
\dim_H E_\mu(\alpha)=\beta^*(\alpha):=\inf_{q\in\RR}\{\alpha q+\beta(q)\},
\end{equation*}
where $\beta(q)$ is the unique number satisfying $\sum_{j=1}^k \pi_j^q r_j^{\beta(q)}=1$. The question is, whether in the above equation we can replace $E_\mu$ with $\overline{E}_\mu$ or $\underline{E}_\mu$. The author could not find an answer to this question in the literature, but in any case, it follows immediately from our results that the answer is affirmative for $\underline{E}_\mu$:

\begin{corollary}
With the notation and assumptions of Example \ref{ex:self-similar-measure}, put 
\begin{equation*}
\alpha_{\min}:=\min_{1\leq j\leq k}\frac{\log\pi_j}{\log r_j}, \qquad \alpha_{\max}:=\max_{1\leq j\leq k}\frac{\log\pi_j}{\log r_j}. 
\end{equation*}
Then $\dim_H \underline{E}_\mu(\alpha)=\beta^*(\alpha)$ if $\alpha\in[\alpha_{\min},\alpha_{\max}]$, and $\underline{E}_\mu(\alpha)=\emptyset$ otherwise.
\end{corollary}

\begin{proof}
It is a routine exercise to show that, if $\alpha=\underline{\alpha}_\mu(t)$, then for every $\eps>0$,
\begin{equation*}
\limsup_{r\downarrow 0}\frac{\mu(B(t,r))}{r^{\alpha+\eps}}=\infty\qquad\mbox{and}\qquad \limsup_{r\downarrow 0}\frac{\mu(B(t,r))}{r^{\alpha-\eps}}=0.
\end{equation*}
Thus,
\begin{equation*}
\underline{\alpha}_\mu(t)=\sup\left\{\alpha>0: \limsup_{r\downarrow 0}\frac{\mu(B(t,r))}{r^{\alpha}}=0\right\}.
\end{equation*}
Let $f(t):=\mu([0,t])$. Then $f$ is of the form \eqref{eq:our-self-affine-functions-intro}. Since $|f(t+h)-f(t)|\leq \mu(B(t,h))\leq |f(t+h)-f(t)|+|f(t-h)-f(t)|$, we see that $\underline{\alpha}_\mu(t)=\tilde{\alpha}_f(t)$. The result now follows immediately from Theorem \ref{thm:Holder-spectrum}.
\end{proof}

In fact, Theorems \ref{thm:individual-Holder-exponents} and \ref{thm:pointwise-Holder-at-endpoints} give the precise value of $\underline{\alpha}_\mu(t)$ at each point $t$.

Unfortunately, the results and analysis of this paper have no direct implications for the Hausdorff dimension of $\overline{E}_\mu(\alpha)$. (However, from the main result of \cite{LiWuXiong} it follows that $\dim_H \overline{E}_\mu(\alpha)=\beta^*(\alpha)$ for all $\alpha\leq\hat{\alpha}$.)
It also should be noted that we obtain $\dim_H \underline{E}_\mu(\alpha)=\beta^*(\alpha)$ only for self-similar measures in $\RR$, and our method has no obvious generalization to higher dimensions.

\section{Monofractal functions and time subordinators} \label{sec:CMT}

Return now to the general function $f$ of the form \eqref{eq:our-self-affine-functions-intro}, with arbitrary $m$, $d$, contraction ratios $c_1,\dots,c_m$ and $\lambda_1,\dots,\lambda_m$, and signature $\boldsymbol{\eps}$. Seuret \cite{Seuret} gave conditions under which a continuous function can be written as the composition of a monofractal function and an increasing function, or time subordinator. We show here, as a consequence of our main results, that for $f$ of the form \eqref{eq:our-self-affine-functions-intro} such a decomposition is always possible. 

Let $s\geq 1$ be the solution of $\sum_{i=1}^m \lambda_i^s=1$. There is a unique monotone increasing function $g:[0,1]\to[0,1]$ of the form \eqref{eq:our-self-affine-functions-intro} satisfying $g(0)=0$, $g(1)=1$, with $m(g)=m$, $c_i(g)=c_i$ and $\lambda_i(g)=\lambda_i^s$ for $i=1,\dots,m$, and $\boldsymbol{\eps}(g)=\boldsymbol{\eps}$.

 
It is a direct consequence of Theorems \ref{thm:individual-Holder-exponents} and \ref{thm:pointwise-Holder-at-endpoints} that
\begin{equation*}
\tilde{\alpha}_g(t)=s\tilde{\alpha}_f(t) \qquad\mbox{for all $t\in(0,1)$},
\end{equation*}
so the nondirectional H\"older spectrum of $g$ is just a horizontal scaling of that of $f$: $\dim_H \tilde{E}_{g}(\alpha)=\dim_H \tilde{E}_f(\alpha/s)$. When $c_1=\dots=c_m=1/m$ and $\boldsymbol{\eps}=(0,0,\dots,0)$, we may replace $\tilde{\alpha}$ with $\alpha$ and $\tilde{E}$ with $E$ in the above. 

Now define $h:[0,1]\to\RR^d$ by $h:=f\circ g^{-1}$, where $g^{-1}(y):=\inf\{t\in[0,1]:g(t)\geq y\}$ for $0\leq y\leq 1$. Since $f$ and $g$ are continuous, $f=h\circ g$. Furthermore, $h$ is also of the form \eqref{eq:our-self-affine-functions-intro}, with $m(h)=\#\II_+$, $c_i(h)=\lambda_{j(i)}^s$ and $\lambda_i(h)=\lambda_{j(i)}$, where $j(i)$ is the index of the the $i$th nonzero entry of the vector $(\lambda_1,\dots,\lambda_m)$. It follows immediately from Theorem \ref{thm:individual-Holder-exponents} that $h$ is monofractal with constant H\"older exponent $\alpha_h(t)=1/s$ for all $t\in(0,1)$. 

Since $g$ is monotone increasing, it can be viewed as a time subordinator for the function $f$. For example, in the case of the P\'olya curve we have $s=2$, and the subordinator $g$ is the Riesz-Nagy function with parameter $a=\sin^2\theta$. The function $h$ in this case is a reparametrization of the P\'olya curve which fills equal areas in equal time, and has constant H\"older exponent $1/2$. Similarly, the time subordinator for Okamoto's function with parameter $a>1/2$ (i.e. the non-monotone case) is the Okamoto function with parameter $a^s$, where $s$ is the unique root of $2a^s+(2a-1)^s=1$. (This example was given by Seuret.) Finally, observe that the time subordinator $g$ is always a singular function, in view of Proposition \ref{prop:zero-or-none}.

\section*{Acknowledgment}
The author wishes to thank Prof. Lars Olsen for bringing the paper \cite{LiWuXiong} to his attention, Prof. Zolt\'an Buczolich for a helpful discussion about pointwise H\"older exponents, and the referee for his or her careful reading of the paper and for suggesting several improvements to the presentation.


\footnotesize
\begin{thebibliography}{23}

\bibitem{Allaart}
{\sc P. C. Allaart}, The infinite derivatives of Okamoto's self-affine functions: an application of $\beta$-expansions. {\em J. Fractal Geom.} {\bf 3} (2016), no. 1, 1--31.

\bibitem{Allaart2} 
{\sc P. C. Allaart}, Differentiability of a two-parameter family of self-affine functions. {\em J. Math. Anal. Appl.} {\bf 450} (2017), no. 2, 954--968.

\bibitem{Arbeiter}
{\sc M. Arbeiter} and {\sc N. Patzschke}, Random self-similar multifractals. {\em Math. Nachr.} {\bf 181} (1996), 5--42.

\bibitem{BKK}
{\sc B. B\'ar\'any, G. Kiss} and {\sc I. Kolossv\'ary}, Pointwise regularity of parameterized affine zipper fractal curves. Preprint, arXiv:1608.04558.

\bibitem{BSS}
{\sc L. Barreira, B. Saussol} and {\sc J. Schmeling}, Distribution of frequencies of digits via multifractal analysis. {\em J. Number Theory} {\bf 97} (2002), 410--438.

\bibitem{Bedford}
{\sc T. Bedford}, H\"older exponents and box dimension for self-affine fractal functions. {\em Constr. Approx.} {\bf 5} (1989), 33--48.

\bibitem{BenSlimane}
{\sc M. Ben Slimane}, Multifractal formalism for selfsimilar functions expanded in singular basis. {\em Appl. Comput. Harmon. Anal.} {\bf 11} (2001), 387--419.




\bibitem{Bumby}
{\sc R. T. Bumby}, The differentiability of P\'olya's function. {\em Adv. Math.} {\bf 18} (1975), 243--244.

\bibitem{Eggleston}
{\sc H. Eggleston}, The fractional dimension of a set defined by decimal properties. {\em Quart. J. Math. Oxford Ser.} {\bf 20} (1949), 31--36.

\bibitem{Falconer}
{\sc K.~J.~Falconer}, {\it Fractal Geometry. Mathematical Foundations and Applications}, 2nd Edition, Wiley (2003)

\bibitem{Frisch}
{\sc U. Frisch} and {\sc G. Parisi}, Fully developed turbulence and intermittency, in Proc. Enrico Fermi, International Summer School in Physics, North Holland, Amsterdam (1985), 84-88.

\bibitem{Jaffard1}
{\sc S. Jaffard}, Multifractal formalism for functions part I: results valid for all functions. {\em SIAM J. Math. Anal.} {\bf 28} (1997), no. 4, 944--970.

\bibitem{Jaffard2}
{\sc S. Jaffard}, Multifractal formalism for functions part II: self-similar functions.  {\em SIAM J. Math. Anal.} {\bf 28} (1997), no. 4, 971--998.

\bibitem{JafMan}
{\sc S. Jaffard} and {\sc B. B. Mandelbrot}, Local regularity of nonsmooth wavelet expansions and application to the Polya function. {\em Adv. Math} {\bf 120} (1996), 265--282.

\bibitem{Katsuura} 
{\sc H. Katsuura}, Continuous nowhere-differentiable functions - an application of contraction mappings. {\em Amer. Math. Monthly} {\bf 98} (1991), no. 5, 411--416.


\bibitem{Kobayashi} 
{\sc K. Kobayashi}, On the critical case of Okamoto's continuous non-differentiable functions. {\em Proc. Japan Acad. Ser. A Math. Sci.} {\bf 85} (2009), no. 8, 101--104.

\bibitem{Kobayashi-Z} 
{\sc Z. Kobayashi}, Digital sum problems for the Gray code representation of natural numbers. {\em Interdiscip. Inform. Sci.} {\bf 8} (2002), 167--175.

\bibitem{Lax}
{\sc P. D. Lax}, The differentiability of P\'olya's function. {\em Adv. Math.} {\bf 10} (1973), 456--464.

\bibitem{Levy}
{\sc P. L\'evy}, Les courbes planes ou gauches et les surfaces compos\'ees de parties semblables au tout. {\em J. Ecole Polytechn.} (1938), 227-247, 249-291.

\bibitem{LiWuXiong}
{\sc J. Li}, {\sc M. Wu} and {\sc Y. Xiong}, Hausdorff dimensions of the divergence points of self-similar measures with the
open set condition. {\em Nonlinearity} {\bf 25} (2012), 93--105.

\bibitem{Li-Dekking}
{\sc W. Li} and {\sc F. M. Dekking}, Hausdorff dimension of subsets of Moran fractals with prescribed group frequency of their codings. {\em Nonlinearity} {\bf 16} (2003), 187--199.

\bibitem{Mandelbrot}
{\sc B. Mandelbrot}, Intermittent turbulence in self-similar cascades: Divergence of high moments
and dimension of the carrier. {\em J. Fluid Mech.} {\bf 62} (1974), 331.

\bibitem{Okamoto} 
{\sc H. Okamoto}, A remark on continuous, nowhere differentiable functions. {\em Proc. Japan Acad. Ser. A Math. Sci.} {\bf 81} (2005), no. 3, 47--50.

\bibitem{Perkins} 
{\sc F. W. Perkins}, An elementary example of a continuous non-differentiable function. {\em Amer. Math. Monthly} {\bf 34} (1927), 476--478.

\bibitem{Polya}
{\sc G. P\'olya}, \"Uber eine Peanosche Kurve. {\em Bull. Acad. Sci. Cracovie} {A} (1913), 305-313.

\bibitem{deRham}
{\sc G. de Rham}, Sur quelques courbes definies par des \'equations fonctionnelles. {\em Univ. e Politec. Torino. Rend. Sem. Mat.} {\bf 16} (1957), 101--113.

\bibitem{Riesz-Nagy}
{\sc F. Riesz} and {\sc B. Sz.-Nagy}, {\em Functional Analysis}, Ungar, New York, 1955.

\bibitem{Sagan}
{\sc H. Sagan}, {\em Space-filling curves}, Springer-Verlag, New York, 1994.

\bibitem{Salem}
{\sc R. Salem}, On some singular monotonic functions which are strictly increasing. {\em Trans. Amer. Math. Soc.} {\bf 53} (1943), 427--439.

\bibitem{Seuret}
{\sc S. Seuret}, On multifractality and time subordination for continuous functions. {\em Adv. Math.} {\bf 220} (2009), no. 3, 936--963.

\end{thebibliography}
\end{document}